\documentclass{article}
\usepackage[utf8]{inputenc}

\usepackage[top=.65in,left=1in,right=1in,bottom=1in]{geometry}
\usepackage{amsmath}
\usepackage{amsfonts}
\usepackage{amssymb}
\usepackage{amsthm}
\usepackage{enumitem}
\usepackage{verbatim}
\usepackage{slashed}
\usepackage{multicol}
\usepackage{url}
\usepackage{graphicx}
\graphicspath{ {images/} }
\usepackage{tikz}

\newtheorem{theorem}{Theorem}[section]
\newtheorem{lemma}[theorem]{Lemma}
\newtheorem{proposition}[theorem]{Proposition}
\newtheorem{corollary}[theorem]{Corollary}
\newtheorem{definition}[theorem]{Definition}

\theoremstyle{definition}

\newcommand{\ZZ}{\mathbb{Z}}
\newcommand{\RR}{\mathbb{R}}

\newcommand{\z}{\mathbf}
\newcommand{\one}{\mathbf 1}

\title{Unicyclic graphs and the inertia of the distance squared matrix}
\author{Christian Howell\thanks{Department of Mathematics, Brigham Young University, Provo, UT (\texttt{cahowell0@gmail.com}).}, Mark Kempton\thanks{Department of Mathematics, Brigham Young University, Provo, UT (\texttt{mkempton@mathematics.byu.edu}).}, Kellon Sandall\thanks{Department of Mathematics, Brigham Young University, Provo, UT (\texttt{kellon08@gmail.com}).}, and John Sinkovic\thanks{Department of Mathematics, Brigham Young University Idaho, Rexburg, ID (\texttt{sinkovicj@byui.edu}).}}
\date{}

\begin{document}

\maketitle

\begin{abstract}
    A result of Bapat and Sivasubramanian gives the inertia of the distance squared matrix of a tree.  We develop general tools on how pendant vertices and degree 2 vertices affect the inertia of the distance squared matrix and use these to give an alternative proof of this result.  We further use these tools to extend this result to certain families of unicyclic graphs, and we explore how far these results can be extended.
\end{abstract}


\section{Introduction}

In spectral graph theory, we use the spectra of matrices associated to a graph to study properties of the graph.  There are many matrices that capture various graph properties, such as the adjacency matrix and the combinatorial Laplacian, which are well-studied.  A matrix that has not received as much attention is the \emph{distance squared matrix} $\Delta$.  For a graph on $n$ vertices, $\Delta$ is the $n\times n$ matrix whose entry $\Delta_{ij}$ is the square of the distance from vertex $i$ to vertex $j$ in the graph.  

When the graph is a tree, Bapat and Sivasubramanian \cite{BAPAT2016328} proved a remarkable connection between the graph and the inertia of its distance squared matrix.  Recall that the \emph{inertia} of a Hermitian $M$ is defined to be the triple $(i_+(M),i_-(M),i_0(M))$ where $i_+$, $i_-$, and $i_0$ are defined, respectively, to be the number (counting multiplicity) of positive eigenvalues, negative eigenvalues, and the multiplicity of 0 as an eigenvalue of $M$.  In \cite{BAPAT2016328} we find the following theorem.

\begin{theorem}[{\cite[Theorem 20]{BAPAT2016328}}]\label{thm:bapat}
Let $T$ be a tree with distance squared matrix $\Delta$, let $\ell$ be the number of vertices of degree 1 in $T$, and let $t$ be the number of vertices of degree 2 in $T$.  Then 
\[
i_0(\Delta)=\begin{cases}t-1 & \text{if } t\geq1\\0 &\text{if }t=0 \end{cases}
\]
and 
\[
i_-(\Delta)=\ell.
\]
\end{theorem}
While many theorems in spectral graph theory bound some graph parameter in terms of eigenvalues, Theorem \ref{thm:bapat} is striking in that only the inertia of a matrix is needed to get an exact enumeration of certain kinds of vertices in the tree.

The purpose of this paper is to extend Theorem \ref{thm:bapat} to families of graphs other than trees.  A natural first step is to consider \emph{unicyclic graphs}, that is, graphs with only one cycle.  As we will see below, while there is no single, unifying theorem characterizing the inertia of all unicyclic graphs, there are various families of unicyclic graphs where results analogous to Theorem \ref{thm:bapat} hold.

The remainder of the paper is organized as follows.  We begin Section \ref{sec:prelim} with needed matrix theoretic tools and preliminaries.  Then in Section \ref{sec:general} we prove some technical lemmas about degree 1 and 2 vertices that apply to any graph.  We then use these tools in Section \ref{sec:tree} to give an alternative, simpler proof to Theorem \ref{thm:bapat}.  Then we begin our investigation of unicyclic graphs in Section \ref{sec:cycle} by giving the complete spectrum (and hence the inertia) of the distance squared matrix of any cycle.  Section \ref{sec:1even} contains the main result of the paper, which is that if $U$ is a unicyclic graph whose cycle is of even length $2q$ and whose cycle contains only one vertex of degree three in $U$, and if $\Delta$ is the distance squared matrix of $U$, then $i_-(\Delta)=\ell+q$ and $i_0(\Delta)=t$.  This is analogous to Theorem \ref{thm:bapat}.  In Section 7, we study several other families of unicyclic graphs that show that this result cannot be extend to all unicyclic graphs.  We end in Section \ref{sec:conc} with a discussion of open questions.

\section{Preliminaries}\label{sec:prelim}

Let $G$ be a connected graph with vertex set $V(G) = \{1, \ldots, n\}$ and edge set $E(G)$. We define the distance between vertices $i, j \in V(G)$, $d_{ij}$, as the minimum number of edges in a path from $i$ to $j$. We will set $d_{ii} = 0$ for $i = 1, \ldots, n$. We will define the distance matrix $D(G)$ or $D$ of the graph as the matrix where each $(i, j)$-entry is equal to $d_{ij}$. Throughout this paper we will be mainly using the distance squared matrix, $\Delta$, which we will define as the Hadamard product $D \circ D$. On occasion, we will also use the notation $A(i)$ which refers to some matrix $A$ with the $i$th column and $i$th row deleted.  Let $\one$ denote the vector whose entries are all equal to 1, and $J_{n}$ will denote the $n\times n$ matrix whose entries are all 1.

% This is my favorite rendition of the Cauchy Interlacing Theorem.
We list here some tools from matrix theory that we will use. We will use the Cauchy Interlacing Theorem extensively, which is

\begin{theorem}[Theorem 4.3.28 of \cite{horn2012matrix}]
Suppose $A \in \RR^{n \times n}$ is symmetric. Let $B \in \RR^{m \times m}$ with $m < n$ be a principal submatrix (obtained by deleting both the $i$-th row and $i$-th column for some values of $i$). Suppose $A$ has eigenvalues $\lambda_1 \leq \cdots \leq \lambda_n$ and $B$ has eigenvalues $\beta_1 \leq \cdots \leq \beta_m$. Then
\begin{equation*}
    \lambda_k \leq \beta_k \leq \lambda_{k + n - m} \quad \quad \text{for} \quad \quad k = 1, \ldots, m
.\end{equation*}
If $m = n - 1$,
\begin{equation*}
    \lambda_1 \leq \beta_1 \leq \lambda_2 \leq \beta_2 \leq \cdots \leq \beta_{n - 1} \leq \lambda_n
.\end{equation*}
\end{theorem}

The following definition contains a shorthand for expressing inertia for the remainder of the paper. 

\begin{definition}%[Inertia]
The inertia, $i(M)$, of a matrix $M$ is the triple 
\begin{equation*}
	i(M) = (i_+(M), i_-(M), i_0(M))
\end{equation*}
where $i_+$, $i_-$, and $i_0$ represent the number of positive, negative, and zero eigenvalues respectively. 
\end{definition} 

% I can't find this theorem anywhere on the internet, but I'm pretty sure it's true.
We will also use the subadditivity property of inertia which is

\begin{theorem}[Proposition 3.9 of \cite{inertia13001}]
Let $A, B \in \RR^{n \times n}$ be symmetric. Then we have 
\begin{equation*}
    i_+(A + B) \leq i_+(A) + i_+(B)
\end{equation*}
and
\begin{equation*}
    i_-(A + B) \leq i_-(A) + i_-(B),
\end{equation*}
where $i_+$ and $i_-$ denote the number of positive and negative eigenvalues respectively.
\end{theorem}

We will use Sylvester's Law of Inertia in many of the row and column operations we perform:

\begin{theorem}[Theorem 4.5.8 in \cite{horn2012matrix}]\label{thm:sylvester}
	Let $A$ be a symmetric $n \times n$ matrix and let $S$ be a nonsingular $n \times n$ matrix. Then $i(A) = i(S^T A S)$. 
\end{theorem}

We remark that $A$ and $S^TAS$ are called \emph{congruent} matrices, so Sylvester's Law of Inertia states that congruent matrices have the same inertia. We also have the Haynsworth inertia additivity formula:

\begin{theorem}[Theorem 4.5.P21 in \cite{horn2012matrix}]\label{thm:haynsworth_additivity_formula}
    Let $H$ be a Hermitian matrix given by
    \begin{equation*}
        H = \begin{bmatrix}
            H_{11} & H_{12} \\
            H_{12}^* & H_{22}
        \end{bmatrix},
    \end{equation*}
    where $H_{11}$ is nonsingular and $H_{12}^*$ is the conjugate transpose of $H_{12}$. Then
    \begin{equation*}
        i(H)=i(H_{11}) + i(H_{22} - H_{12}^*H_{11}^{-1}H_{12})
    .\end{equation*}
\end{theorem}

When we deal with cycles, we will often use circulant matrices which are matrices in which the first row is repeated for each row of the matrix except shifted one entry to the right. They will take the form
\begin{equation*}
    \begin{bmatrix}
        a_1 & a_2 & \cdots & a_{n - 1} & a_n \\
        a_n & a_1 & \cdots & a_{n - 2} & a_{n - 1} \\
        \vdots & \vdots & \ddots & \vdots & \vdots \\
        a_2 & a_3 & \cdots & a_n & a_1 \\
    \end{bmatrix}
.\end{equation*}
To abbreviate, we will also express circulant matrices as
\begin{equation*}
    \begin{bmatrix}
        a_1 & a_2 & \cdots & a_{n - 1} & a_n \\
        & \ddots & \ddots & \ddots
    \end{bmatrix}
\end{equation*}
throughout the paper. The following lemma gives the eigenvectors and eigenvalues of circulant matrices.

% Maybe have a Lemma here about the eigenvalues of circulant matrices.
\begin{lemma}[Theorem 2.2.P10 of \cite{horn2012matrix}]\label{lem:eigenstuff_of_circulant_matrices}
    Let $A$ be the circulant matrix above. The eigenvectors will take the form
    \begin{equation*}
        \mathbf{v}_j = \frac 1 {\sqrt{n}} \begin{bmatrix}
            1 & \omega^j & \omega^{2j} & \cdots & \omega^{(n - 1)j}
        \end{bmatrix}^\top
        \quad \quad \quad
        j = 1, 2, \ldots, n,
    \end{equation*}
    where $\omega = \exp\left(\frac{2\pi i}{n}\right)$ is a primitive $n$-th root of unity and $i$ is the imaginary unit.
    
    The eigenvalues will take the form 
    \begin{equation*}
        \lambda_j = a_1 + a_2 \omega^j + \cdots + a_n \omega^{(n - 1)j}
        \quad \quad \quad
        j = 1, 2, \ldots, n.
    \end{equation*}
\end{lemma}

\section{Degree 1 and 2 Vertices}\label{sec:general}

In this section, we will prove some preliminary results that will apply to general graphs that have multiple degree one vertices that are all neighbors, and degree two vertices that are part of a ``branch" of the graph.

Let $G$ be any graph obtained by attaching $k\geq2$ pendant vertices to a vertex of $w$ of some graph $G'$ and let $\Delta$ be the distance squared matrix of $G$.  We will order the vertices so that vertices $1,...,k$ are the pendant vertices.  Denote by $\Delta_{k+1}$ the distance squared matrix of $G\backslash\{1,...,k\}=G'$ and $\Delta_k$ the distance squared matrix of $G\backslash \{1,...,k-1\}$. 

\begin{proposition}\label{lem:kpend}
 In the above notation, we have
 \[
 i_-(\Delta) = k-1 + i_-\left(\Delta_{k} + \left(4-\frac{4}{k}\right)\mathbf{e}_k\mathbf{e}_k^T\right)
 \]
 where $e_k$ denotes the indicator vector for the $k$th vertex.
\end{proposition}
\begin{proof}
We have

\[\Delta = \begin{bmatrix}
    4J_k-4I_k & X^T\\
    \\
    X & \Delta_{k + 1}
\end{bmatrix},
\]
    where $\Delta_{k+1}$ is the square distance matrix for $G \backslash \{1,...,k\}$ and $X$ is a $(n - k) \times k$ matrix with all identical columns.

We will proceed by performing row operations and the corresponding column operations to produce a matrix that is congruent to $\Delta$ (and hence has the same inertia) whose inertia will be easier to see.  Note that performing a row operation and its corresponding column operation corresponds to multiplying the matrix by an elementary matrix on the left and its transpose on the right, and is thus a congruence.  

First we subtract row $k$ and column $k$ from each of the first $k-1$ rows and columns.  This will zero out the first $k-1$ rows of $X^T$ and the first $k-1$ columns of $X$ and will leave us with $-4I_{k-1}-4J_{k-1}$ in the upper left corner. We then subtract $1/k$ times each of the first $k-1$ rows from row $k$, and similarly to the columns.  We then find that $\Delta$ is congruent to

\[A = \begin{bmatrix}
    -4I_{k-1}-4J_{k-1} & 0\\
    \\
    0 & (4 - {4/k})\mathbf{e}_k\mathbf{e}_k^T + \Delta_k
\end{bmatrix},\]
with $\Delta_k$ being the square distance matrix for $G \backslash \{1,...,k-1\}$. Notice that $-4I_{k-1}-4J_{k-1}$ is negative definite and so has $k-1$ negative eigenvalues.  Then by Sylvester's Law of Inertia (Thereom \ref{thm:sylvester}), $i(\Delta)=i(A)$ and the result follows.
\end{proof}

\begin{corollary}
In the above setting, we have that either $i_-(\Delta) = k-1 + i_-(\Delta_{k+1})$ or $i_-(\Delta)=k+i_-(\Delta_{k+1})$.
\end{corollary}
\begin{proof}
Since $\Delta_{k+1}$ results from deleting a row and column of $\Delta_k+(4-4/k)\mathbf{e}_k\mathbf{e}_k^T$, then the result follows from Proposition \ref{lem:kpend} and interlacing.
%Note that the matrix $\left(4-\frac{4}{k}\right)\mathbf{e}_k\mathbf{e}_k^T$ that shows up in the previous result is positive semidefinite
\end{proof}

In summary, adding $k\geq2$ pendant vertices to a single vertex will increase the number of negative eigenvalues of the distance squared matrix by at least $k-1$, and possibly by $k$.  Examples of both can occur, as we will see, and much of the work in what follows will be to determine if the increase is by $k-1$ or $k$ in specific kinds of graphs.

\begin{lemma}\label{lem:tree_null} For a graph $G$ with distance squared matrix $\Delta$, let $v$ be a vertex of degree $2$ such that its removal disconnects the graph.  Let $u$ and $w$ be the neighbors of $v$, and let $\mathbf{x}$ be a vector in $\mathbb{R}^n$ such that 

\[ x_i = \begin{cases}
    1 & \text{ if } i \text{ corresponds to vertex } u, w \\
    -2 & \text{ if } i \text{ corresponds to vertex } v\\
    0 & \text{ otherwise.} 
    \end{cases}
\]
Then $\Delta \mathbf{x}=2\cdot\one$.
\end{lemma}

\begin{proof}
Let $G$ be a graph with distance squared matrix $\Delta$. Suppose that $v$ is vertex of degree 2 with neighbors $u$ and $w$. Let $\mathbf{x}$ be a vector in $\RR^n$ defined as above. 
Suppose also that $G$ can be disconnected by removing $(u, v)$ or $(v, w)$. Each row of $\Delta$ will take the form $\begin{bmatrix}\cdots & j^2 & (j + 1)^2 & (j + 2)^2 & \cdots\end{bmatrix}$ where $j$ depends on the location of the vertex in the graph with which the row corresponds and the positions of $j^2$, $(j + 1)^2$, and $(j + 2)^2$ correspond with the positions of $u$, $v$, and $w$ in G respectively (note that $j$ itself could be negative). The vector $\mathbf{x}$ will take the form $\begin{bmatrix}\mathbf{0} & 1 & -2 & 1 & \mathbf{0}\end{bmatrix}^\top$ where the $1$, $-2$, and $1$ appear in the same position as $j^2$, $(j + 1)^2$, and $(j + 2)^2$, respectively. Multiplying an arbitrary row of $\Delta$ with $\mathbf{x}$ gives
\begin{align*}
    \begin{bmatrix}\cdots & j^2 & (j + 1)^2 & (j + 2)^2 & \cdots\end{bmatrix}
    \begin{bmatrix}\mathbf{0} \\ 1 \\ -2 \\ 1 \\ \mathbf{0}\end{bmatrix}
    &= j^2 - 2(j + 1)^2 + (j + 2)^2 \\
    &= j^2 - 2j^2 - 4j - 2 + j^2 + 4j + 4 \\
    &= 2
.\end{align*}
%Thus $\Delta \mathbf{x} = 2\cdot\one$.
\end{proof}

\begin{lemma}\label{lem:tree_lin_ind}
 Define a vector $\mathbf{x}_i$ for each vertex of degree $2$ in the same fashion as the lemma above. Then $\{\mathbf{x}_i\}$ is a linearly independent set.
\end{lemma}

\begin{lemma}\label{lem:deg_2_verts} 
If $G$ has $t \geq 2$ vertices of degree $2$ whose removal disconnects $G$, then $i_0 (\Delta) \geq t-1$.
\end{lemma}

\begin{proof}
Let $\{\mathbf{x}_1, ... , \mathbf{x}_t\}$ be as in Lemma \ref{lem:tree_lin_ind}. Note that $\{\mathbf{x}_1 - \mathbf{x}_2, \mathbf{x}_1 - \mathbf{x}_3, ... , \mathbf{x}_1 - \mathbf{x}_t\}$ is a linearly independent set from the lemma above. Also $\Delta(\mathbf{x}_1 - \mathbf{x}_2) = \Delta \mathbf{x}_1 - \Delta \mathbf{x}_2 = 0.$  Thus each vector is in the null space of $\Delta$ and so $i_0(\Delta) \geq t - 1$. 
\end{proof}

\section{Trees}\label{sec:tree}

In this section we will apply the results of the previous section to determine the inertia of distance  squared matrix of a tree. We let $T$ be a tree and $\Delta$ be its distance squared matrix. %We will use the following results.

The following is the main result of this section.

\bigskip
\begin{theorem}\label{thm:tree_inertia} Let $T$ be a tree of $n > 2$ vertices and $\Delta$ be the distance squared matrix of $T$. Then 

\[ i_0(\Delta) = \begin{cases}
    0 & \text{ if } t = 0,1\\
    t - 1 & \text{ if } t \geq 2,
\end{cases}
\]
where $t$ is the number of degree $2$ vertices of $T$ and 
\[i_-(\Delta)=\ell,\]
where $\ell$ is the number of leaves (vertices of degree one) of $T$.
\end{theorem}

\begin{proof}
We proceed by induction on the number of vertices $n$. For $n = 3$, $T\cong P_3$, and 
\[\Delta = \begin{bmatrix}
    0 & 1 & 4\\
    1 & 0 & 1\\
    4 & 1 & 0
\end{bmatrix},\]
with eigenvalues of $2 + \sqrt{6}, -4, 2 - \sqrt{6}$, satisfying the conclusion of the theorem.

Now, let $T$ be a tree of $n > 3$ vertices and assume the conclusion is true for all trees of fewer than $n$ vertices (but at least 3 vertices). If diam$(T) = 2$, then $T$ is a star and 
\[\Delta = \begin{bmatrix}
    0 & \textbf{1}^T\\
    \textbf{1} & 4J - 4I
\end{bmatrix},\]
where $4J - 4I$ is $(n - 1) \times (n - 1)$, and $J$ is the all ones matrix. We have
\[i_-(4J - 4I) = n - 2 \text{ (where the eigenvalue equals $-4$ with multiplicity of $n - 2$)},\]
\[i_+(4J - 4I) = 1 \text{ (where the eigenvalue equals $4(n - 2)$)}.\]

Performing similar row and column operations we subtract from the first row $1 / (4(n - 2))$ times each of the remaining rows of $\Delta$ and from the first column $1 / (4(n - 2))$ times each of the remaining columns of $\Delta$. Thus, $\Delta$ is congruent to
\[\begin{bmatrix}
    -\cfrac{n-1}{4(n-2)} & \textbf{0}^T\\
    \\
    \textbf{0} & 4J - 4I
\end{bmatrix}.
\]
Therefore, by Sylvester's law of inertia, $i_-(\Delta) = n - 1$ and $i_0 (\Delta) = 0$. Hence, from here on we assume diam$(T) \geq 3$.

Consider a diametrical path in $T$. There exists a vertex $w$ such that $w$ has exactly one non-leaf neighbor $u$.

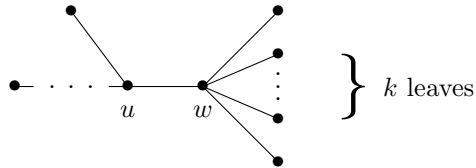
\begin{center}
\begin{tikzpicture}
\draw (-{3/2},0)node{$\bullet$}--(-{5/4},0) (-{3/4},0)node{. . .} (-{1/4},0)--(0,0) (-{3/4},1)node{$\bullet$}--(0,0)node{$\bullet$}--(1,0)node{$\bullet$} (1,-{1/3})node{$w$} (0,-{1/3})node{$u$} (1,0)--(2,1)node{$\bullet$} (1,0)--(2,{3/7})node{$\bullet$} (1,0)--(2,-{3/7})node{$\bullet$} (1,0)--(2,-1)node{$\bullet$} (2,0)node{.} (2,{1/6})node{.} (2,-{1/6})node{.} (3,0)node{\Huge \}} (4,0)node{$k$ leaves};
\end{tikzpicture}
\end{center}
Let $w$ be adjacent to $k$ leaves, and label leaves $1,...,k$. Now let $w = k+1$ and $u = k+2$.

\begin{description}

\item[Case 1:] $k \geq 2$. By Proposition \ref{lem:kpend} we have that $i_-(\Delta) = k-1 + i_-\left(\Delta_{k} + \left(4-\frac{4}{k}\right)\mathbf{e}_k\mathbf{e}_k^T\right)$.  Let $F = (4 - 4/k) \mathbf{e}_k \mathbf{e}_k^T + \Delta_k$ and notice $F(1)=\Delta_{k+1}$ where $\Delta_{k+1}$ is the distance squared matrix for $T \backslash \{1,...,k\}$. Notice that $T \backslash \{1,...,k\}$ is a tree with $\ell - (k - 1)$ leaves and at least 3 vertices. Using our inductive hypothesis on $T \backslash \{1,...,k\}$ to find $i_-(\Delta_k)$, we see that
\begin{equation*}
i_-(\Delta_{k + 1}) = \ell - (k - 1).
\end{equation*}
It follows from the Cauchy interlacing theorem that
\[i_-(F) \geq \ell - (k - 1).\]
Since $F = (4 - 4/k) \boldsymbol{e}_k \boldsymbol{e}_k^T + \Delta_k$ and $4 - 4/k > 0$, the subadditivity of inertia gives us 
\begin{align*}
    i_-(F) &\leq i_-((4-4/k) \boldsymbol{e}_k \boldsymbol{e}_k^T) + i_-(\Delta_k) \\
    &= 0 + \ell - (k - 1) \\
    &= \ell - (k - 1),
\end{align*}
and we conclude that 
$i_-(F) = \ell - (k - 1).$
Thus,
\begin{align*}
    i_-(\Delta) &=k-1 + i_-(F) \\
    &= k - 1 + \ell - (k - 1) \\
    &= \ell.
\end{align*}

\item[Case 2:]$k = 1$. Label the leaf adjacent to $w$ as $v$. Let $\Delta_2 = \Delta(1)$. By the inductive hypothesis, $i_-(\Delta_2) = \ell$. Thus, by interlacing, $i_-(\Delta) \geq \ell$. As there are two ends in a diametrical path, we may assume that there exists a vertex $y$ such that $y$ is adjacent to exactly one leaf $x$ and exactly one non-leaf $z$. If this were not the case we would be able to apply the argument from Case 1.

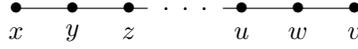
\begin{center}
\begin{tikzpicture}

\draw (0,0)node{$\bullet$}--(3/4,0)node{$\bullet$}--(6/4,0)node{$\bullet$}--(7/4,0) (9/4,0)node{. . .} (11/4,0)--(12/4,0)node{$\bullet$}--(15/4,0)node{$\bullet$}--(18/4,0)node{$\bullet$} (0,-1/3)node{$x$} (3/4,-1/3)node{$y$} (6/4,-1/3)node{$z$} (12/4,-1/3)node{$u$} (15/4,-1/3)node{$w$} (18/4,-1/3)node{$v$};
\end{tikzpicture}
\end{center}

It is possible that $u = y$ and $z = w$ in the case that $T \cong P_4$ or that $u = z$ in the case that diam($T$) $= 4$. In any case, using Lemma \ref{lem:deg_2_verts}, the columns of $\Delta$ corresponding to $x,y,z$ with weights $1, -2, 1$ form the all 2's vector. Thus the column corresponding to $v$ is a linear combination of the columns corresponding to $x,y,z,u,w$. Thus rank$(\Delta)$ = rank$(\Delta_2)$ and $i_-(\Delta) \leq \ell$.

\end{description}

Now we continue to show that $i_0(\Delta) = t - 1$, where $t \geq 2$ is the number of vertices of degree 2, or, in the case that $t = 0, 1$, that $\Delta$ is invertible. Recall from the beginning of the proof that $n \geq 4$ and diam($T$) $\geq 3$.

\begin{description}

\item[Case 1:] $T$ has a vertex which is adjacent to more than two leaves. Then deleting a leaf from $T$ yields $\Delta_2$, which is a principal submatrix of $\Delta$. Note that the number of degree 2 vertices has not changed. By interlacing, $i_+(\Delta) \geq i_+(\Delta_2)$, and by the first part of the theorem $i_-(\Delta) = i_-(\Delta_2) + 1$. If $t = 0$ or $1$ then, by our inductive hypothesis, $\Delta_2$ is invertible. Thus $i_0(\Delta_2) = 0$. So
\[i_+(\Delta_2) + i_-(\Delta_2) = n - 1\]
and
\[n \geq i_+(\Delta) + i_-(\Delta) \geq i_+(\Delta_2) + i_-(\Delta_2) + 1 = n.\]
Thus $i_0(\Delta) = 0$, as desired.

If $t \geq 2$, then by Lemma \ref{lem:deg_2_verts} we have $i_0(\Delta) \geq t - 1$ and by the inductive hypothesis $i_0(\Delta_2) = t - 1$. So
\[n - 1 = i_+(\Delta_2) + i_-(\Delta_2) + i_0(\Delta_2)\]
and
\[n \geq i_+(\Delta) + i_-(\Delta) + i_0(\Delta) \geq i_+(\Delta_2) + i_-(\Delta_2) + 1 +i_0(\Delta_2) = n.\]
Thus $i_0(\Delta) = i_0(\Delta_2) = t - 1$, as desired. 

\item[Case 2:] Every vertex of $T$ is adjacent to at most one leaf. Since diam$(T) \geq 3$, there exists at least two vertices of degree 2, each adjacent to a leaf (consider a diametrical path). Deleting such a pendant vertex decreases the number of vertices of degree 2 by one and keeps the number of leaves the same. Note that we have rank$(\Delta) = $ rank$(\Delta_2)$ by arguments similar to the above when we proved $i_-(\Delta)=\ell$ in the $k=1$ case above. By interlacing, $i_+(\Delta) \geq i_+(\Delta_2)$ and $i_-(\Delta) \geq i_-(\Delta_2)$. Since $i_+(\Delta) + i_-(\Delta) = \text{rank}(\Delta) = \text{rank}(\Delta_2) = i_+(\Delta_2) + i_-(\Delta_2)$, we have $i_+(\Delta) = i_+(\Delta_2)$ and $i_-(\Delta) = i_-(\Delta_2)$. Thus $i_0(\Delta) = i_0(\Delta_2) + 1 = t - 2 + 1 = t - 1$ and the proof is complete.

\item[Case 3:] $T$ has a vertex which is adjacent to exactly two leaves. We consider the vertex $w$ with 2 adjacent leaves. Then the distance squared matrix has the form 

\begin{equation*}
    \Delta = \begin{bmatrix}
        0 & 4 & \textbf{x}^T \\
        4 & 0 & \textbf{x}^T \\
        \textbf{x} & \textbf{x} & \Delta''
    \end{bmatrix},
\end{equation*}
where $\Delta''$ is $(n - 2) \times (n - 2)$ and $\textbf{x}$ is a vector of appropriate size. Notice that if we were to remove both leaves, our distance squared matrix would just be $\Delta''$. From above, we know that adding one leaf on to the end of the branch would result in a matrix, call it $\Delta'$, of the form

\begin{equation*}
    \Delta' = \begin{bmatrix}
        0 & \mathbf{x}^T \\
        \mathbf{x} & \Delta'' \\
    \end{bmatrix}
\end{equation*}
that has one more zero eigenvalue than $\Delta''$. Thus the vector $\mathbf{x}$ is in the column space of $\Delta''$. So the vector $\textbf{x}$ can be reduced to the all zeros vector without changing the column space of $\Delta$. Hence rank$\Delta'' =$ rank$\Delta - 2$, since the upper left $2 \times 2$ matrix has a rank of 2. Therefore, adding the two leaves onto our vertex $w$ has no effect on the number of zero eigenvalues of $\Delta$.

\end{description}

\end{proof}

\section{Cycles}\label{sec:cycle}

To continue our journey to understanding unicyclic graphs, we will now consider cycle graphs.

\begin{definition}
Let $C$ be a graph containing three or more nodes. We call $C$ a cycle graph if every node has exactly two edges connected to it.
\end{definition}

We will obtain formulas for the eigenvalues of the distance squared matrix of cycle graphs.  We will treat even and odd length cycles separately. The basic distance squared matrices of these graphs take the form
\begin{equation*}
    \begin{bmatrix}
    c_1 & c_2 & \cdots & c_{n - 1} & c_n \\
    c_n & c_1 & \cdots & c_{n - 2} & c_{n - 1} \\
    \vdots & \vdots &  & \vdots & \vdots \\
    c_2 & c_3 & \cdots & c_n & c_1
    \end{bmatrix}
\end{equation*}
where $c_1 = 0$ and $c_i = c_{n + 2 - i} = (i - 1)^2$ for $i > 1$. This is an example of a circulant matrix. The eigenvalues will be in the form of $c_1 + c_2\omega_j + \cdots + c_n\omega_j^{n - 1}$. 
\begin{theorem} \label{thm:cycle_eigenvalues}
Let $C$ be a cycle graph of length $n$ and $\lambda_j$ be the $j$th eigenvalue of $C$. When $n$ is odd, we have
\begin{equation} \label{eq:odd_cycle_eigenvalues}
\begin{split}
    \lambda_j = 2\cos(\pi j) \left(\cos\left(\frac{(n-2)\pi}{n}j\right)\right. + 4 \cos \left(\frac{(n-4)\pi}{n}j\right)\\ 
    + \cdots +  \frac{(n - 1)^2}{4}\cos\left(\frac{\pi}{n}j\right)\left.\vphantom{\frac12}\right).
\end{split}
\end{equation}
When $n$ is even, we have
\begin{equation} \label{eq:even_cycle_eigenvalues}
\begin{split}
    \lambda_j = 
    2\cos(\pi j)\left(\frac{n^2}{8} \right.+ \cos\left(\frac{(n - 2)\pi}{n}j\right)
    + 4\cos\left(\frac{(n - 4)\pi}{n}j\right)\\
    + \cdots 
    + \frac{(n - 1)^2}{4}\cos\left(\frac{\pi}{n}j\right)\left.\vphantom{\frac12}\right).
\end{split}
\end{equation}
\end{theorem}

\begin{proof}
We consider the case where $n$ is odd. When we substitute our $c_i$'s into the eigenvalue equation, we will get 
\begin{align*}
    \lambda_j &= \omega_j + 4\omega_j^2 + \cdots + 4\omega_j^{n - 2} + \omega_j^{n - 1} \\
    %&= e^{i\frac{2\pi}{n}j} + 4e^{i\frac{4\pi}{n}j} + \cdots + 4e^{i\frac{(2n - 4)\pi}{n}j} + e^{i\frac{(2n - 2)\pi}{n}j} \\
    %\begin{split}
    %&= \cos\left(\frac{2\pi}{n}j\right) + i\sin\left(\frac{2\pi}{n}j\right) + 4\cos\left(\frac{4\pi}{n}j\right) + i4\sin\left(\frac{4\pi}{n}j\right) + \cdots + 4\cos\left(\frac{(2n - 4)\pi}{n}j\right)\\
    %&\qquad + i4\sin\left(\frac{(2n - 4)\pi}{n}j\right) + \cos\left(\frac{(2n - 2)\pi}{n}j\right) + i\sin\left(\frac{(2n - 2)\pi}{n}j\right)\\
    %\end{split}\\
    &= \cos\left(\frac{2\pi}{n}j\right) + 4\cos\left(\frac{4\pi}{n}j\right) + \cdots + 4\cos\left(\frac{(2n - 4)\pi}{n}j\right) + \cos\left(\frac{(2n - 2)\pi}{n}j\right) \\
    \begin{split}
    &= 2\cos\left(\left(\frac{2\pi}{n}j + \frac{(2n - 2)\pi}{n}j\right)\frac{1}{2}\right)\cos\left(\left(\frac{2\pi}{n}j - \frac{(2n - 2)\pi}{n}j\right)\frac{1}{2}\right) \\
    &\qquad + 8\cos\left(\left(\frac{4\pi}{n}j + \frac{(2n - 4)\pi}{n}j\right)\frac{1}{2}\right)\cos\left(\left(\frac{4\pi}{n}j-\frac{(2n - 4)\pi}{n}j\right)\frac{1}{2}\right) + \cdots\\
    &\qquad + \frac{(n - 1)^2}{2}\cos\left(\left(\frac{(n - 1)\pi}{n}j + \frac{(n + 1)\pi}{n}j\right)\frac{1}{2}\right)\cos\left(\left(\frac{(n - 1)\pi}{n}j - \frac{(n + 1)\pi}{n}j\right)\frac{1}{2}\right) \\
    \end{split}\\
    &= 2\cos(\pi j)\cos\left(\frac{(n - 2)\pi}{n}j\right) + 8\cos(\pi j)\cos\left(\frac{(n - 4)\pi}{n}j\right) + \cdots \\
    &\qquad + \frac{(n - 1)^2}{2}\cos(\pi j)\cos\left(\frac{\pi}{n}j\right)\\
    &= 2\cos(\pi j)\left(\cos\left(\frac{(n - 2)\pi}{n}j\right)
    + 4\cos\left(\frac{(n - 4)\pi}{n}j\right)
    + \cdots 
    + \frac{(n - 1)^2}{4}\cos\left(\frac{\pi}{n}j\right) \right).
\end{align*}

Now we will consider the case of an even cycle. In this case, the above derivation will be very similar except there will be a $n^2/8$ in the large quantity of equation \ref{eq:odd_cycle_eigenvalues}:
\begin{equation*}
    \lambda_j = 
    2\cos(\pi j)\left(\frac{n^2}{8} + \cos\left(\frac{(n - 2)\pi}{n}j\right)
    + 4\cos\left(\frac{(n - 4)\pi}{n}j\right)
    + \cdots 
    + \frac{(n - 1)^2}{4}\cos\left(\frac{\pi}{n}j\right) \right).
\end{equation*}

\end{proof}

\begin{corollary}\label{cor:cycle_eigenvalues}
Let $C$ be a cycle graph of length $n > 2$ and $\Delta$ be its distance squared matrix. If $n \equiv 1 \pmod 4$, then $i_+(\Delta) = \frac{n - 1}{2} + 1$ and $i_-(\Delta) = \frac{n - 1}{2}$. If $n \equiv 3 \pmod 4$, then $i_+(\Delta) = \frac{n - 1}{2}$ and $i_-(\Delta) = \frac{n - 1}{2} + 1$. If $n$ is even, then $i_+(\Delta) = \frac n 2$ and $i_-(\Delta) = \frac n 2$.
\end{corollary}

\begin{proof}
In Theorem \ref{thm:cycle_eigenvalues}, $j$ is an integer. However, by thinking of it as a real number, we can glean information about how the values of the eigenvalues of the matrix are connected. So we will now think of (\ref{eq:odd_cycle_eigenvalues}) as an equation on the real numbers. The zeros for this equation will appear when $j = k/2$ for some $k \in \ZZ$. These will be the only zeros for this equation since the second half only becomes zero when all of the individual components are in sync and that happens only at $n/2$. Since the equation is continuous, this means that none of the eigenvalues will be zero. Now we'll use the first derivative test to find the quantity of positive eigenvalues and negative eigenvalues. The derivative with respect to $j$ of this equation is
\begin{align*}
    &-2\pi \sin(\pi j) \left(\cos\left(\frac{(n - 2)\pi}{n}j\right) + 4\cos\left(\frac{(n - 4)\pi}{n}j\right) + \cdots \right) \\
    &+2\cos(\pi j)\left(-\frac{(n - 2)\pi}{n}\sin\left(\frac{(n - 2)\pi}{n}j\right) - \frac{4(n - 4)\pi}{n}\sin\left(\frac{(n - 4)\pi}{n}j\right) - \cdots\right).
\end{align*}
We will only focus on substituting the zeros we found above into the first half of the expression since the second half is 0 at the values of $j = k/2$ for reasons stated above. The $-2\pi \sin(\pi j)$ part of the expression will alternate between positive and negative starting with negative when $j = k/2$. The other half of the expression will be positive at first since the rightmost term will dominate the others until it hits $n/2$ where it will be 0 for reasons stated above. After that it will be negative. 

We will now deal with the case when the number of vertices mod 4 equals 1. When $j = 1$, $\lambda_j$ is between two zeros. Since the derivative is negative at .5 and positive at 1.5, $\lambda_1$ will be negative. The eigenvalues will flip between positive and negative until it gets to the integer after $n/2$. Here there will be another positive eigenvalue in a row and then start alternating again. So there will be $\lceil n / 2 \rceil$ positive eigenvalues and $\lfloor n / 2 \rfloor$ negative eigenvalues.

It will be very similar to the case when the number of vertices mod 4 equals 3. Only in this case there will be a subsequent negative eigenvalue instead of a positive eigenvalue. 

Now for the even case, we'll focus on equation (\ref{eq:even_cycle_eigenvalues}). This equation is similar to equation (\ref{eq:odd_cycle_eigenvalues}) except it has an extra term. This will cause the derivative to strictly alternate between positive and negative values on the roots of the equation. Therefore, there will be $n/2$ positive and $n/2$ negative eigenvalues. 
\end{proof}

\section{An even length cycle with exactly one tree connected to a vertex}\label{sec:1even}

In the previous section, we examined the eigenvalues of a graph which was a single cycle. In this section, we will build upon the previous result and observe what occurs when we attach a tree to one node of the cycle graph. We will build up to this with a few lemmas.

%Begin proof about inverse of even cycles
\begin{lemma}\label{lem:inverse_even_cycle}Given an even cycle with $p$ nodes and distance squared matrix $\Delta$, the inverse of the matrix can be expressed as the circulant matrix 

\begin{equation*}
    \Delta^{-1} = \frac{1}{4\lambda m} \left( 2J +
    \begin{bmatrix}
        \textbf{0}^T & -\lambda & 2\lambda & -\lambda & \textbf{0}^T \\
         & \ddots & \ddots & \ddots &
    \end{bmatrix}
    \right),
\end{equation*} 
where $m = n/2$, $\lambda$ is the largest eigenvalue of $\Delta$, and the \textbf{0} vectors in the first row are size $(n/2) - 1$ and $(n/2) - 2$, respectively, and each row shifting one to the right, forming a circulant matrix.
\end{lemma}

\begin{proof}
For convenience define 
\[B = 
    \begin{bmatrix}
        \textbf{0}^T & -\lambda & 2\lambda & -\lambda & \textbf{0}^T \\
         & \ddots & \ddots & \ddots &
    \end{bmatrix}.\]
We have
\begin{align*}
    I &= \frac{1}{4\lambda m} \left(\lambda
    \begin{bmatrix}
        4m -2 & -2 & -2 & \dots & -2 \\
        -2 & 4m -2 & -2 & \dots & -2 \\
        \vdots & \vdots & \ddots & & \vdots \\
        -2 & -2 & -2 & \dots & 4m -2\\
    \end{bmatrix} 
    + 2\lambda J\right)\\
    \\
    &= \frac{1}{4\lambda m} \left(\lambda
    \begin{bmatrix}
        -2m^2 + 4m - 2 + 2m^2 & -2 & \dots & -2 \\
        \vdots & \ddots & & \vdots \\
        -2 & \dots & -2 & -2 m^2 + 4m - 2 + 2m^2 
    \end{bmatrix}
    + 2\lambda J\right)\\
    \\
    &= \frac{1}{4\lambda m} \left(\lambda
    \begin{bmatrix}
        -(m - 1)^2 + 2m^2 - (m - 1)^2 & -2 & \dots & -2 \\
        \vdots & \ddots & & \vdots \\
        -2 & \dots & -2 & -(m - 1)^2 + 2m^2 - (m - 1)^2 \\
    \end{bmatrix}
    + 2\lambda J\right)\\
    &= \frac{1}{4\lambda m} (\Delta B + 2\lambda J).
\end{align*}

Because $\Delta$ is circulent, the sum of each of the rows is the same, i.e. $\lambda$. So $2\lambda J = 2\Delta J$. Thus 
\begin{align*}
    \frac{1}{4\lambda m} (\Delta B + 2\lambda J) &= \frac{1}{4\lambda m} (\Delta B + 2\Delta J) \\
    &= \frac{1}{4\lambda m}\Delta(B + 2J).
\end{align*}
\end{proof}
\begin{lemma}\label{lem:x_delta_inverse_x_positive}
    Let $\Delta$ be the distance squared matrix of an even cycle with $p$ nodes and a single pendant extending from the cycle. Let $\Tilde{\Delta}$ be all but the last row and column of $\Delta$, and let $\mathbf{x}$ be the first $p$ elements of the last row of $\Delta$. Then
    \[\mathbf{x}^T \Tilde{\Delta}^{-1} \mathbf{x} > 0.\]
        
\end{lemma}

\begin{proof}
    Using Lemma \ref{lem:inverse_even_cycle} we have
        \[\mathbf{x}^T \Tilde{\Delta}^{-1} \mathbf{x}   = \frac{1}{4 \lambda m} \left(2 \mathbf{x}^T J \mathbf{x} + \mathbf{x}^T 
        \begin{bmatrix}
            \textbf{0}^T & -\lambda & 2\lambda & -\lambda & \textbf{0}^T \\
            & \ddots & \ddots & \ddots &
        \end{bmatrix}
        \mathbf{x} \right). \]
    We consider each term individually.
    
    Multiplying $\mathbf{x}^T$ by $J$ results in a vector whose entries are all equal to the sum of the entries of $\mathbf{x}$. The sum of the elements of $\mathbf{x}$ can be written as 
    \begin{align*}
        2\sum_{j = 1}^{m} j^2 + (m + 1)^2 - 1 &= \frac{m(m + 1)(2m + 1)}{3} + (m + 1)^2 - 1\\
        &= \frac{2m^3 + 6m^2 + 7m}{3}
    \end{align*}
    where $m = p / 2$. Hence
        \begin{align*}
            2\mathbf{x}^T J \mathbf{x} &= 2\left(\frac{2m^3 + 6m^2 + 7m}{3} \right)\textbf{1}^T\textbf{x}\\
            &= 2\left(\frac{2m^3 + 6m^2 + 7m}{3}\right)^2.
        \end{align*}
    Now we consider the second term. Let $B$ be as defined in Lemma \ref{lem:inverse_even_cycle}. Letting $j$ denote the index of the element of the vector $\mathbf{x}$ corresponding to the index of the first non-zero entry of the row of the matrix $B$
    by which we are multiplying $\mathbf{x}$, we see that the first entry of $B \mathbf{x}$ is given by
    \begin{align*}
        \begin{bmatrix}
            -\lambda & 2 \lambda & -\lambda
        \end{bmatrix}
        \begin{bmatrix}
            j^2\\
            (j + 1)^2\\
            j^2
        \end{bmatrix} &= -\lambda j^2 + 2 \lambda (j + 1) ^ 2 - \lambda j^2\\
        &= -2 \lambda j^2 + 2 \lambda (j^2 + 2j + 1)\\
        &= 4 \lambda j + 2 \lambda.
    \end{align*}
    Since the index of the first non-zero term will be $m + 1$, the first entry of our resultant matrix will be $4 \lambda (m + 1) + 2 \lambda$. The $m + 1$ entry of our resultant matrix will be given similarly, replacing $m + 1$ with $1$.  This gives us $-6 \lambda$ as the $m + 1$ term in our vector.  Each of the other entries will be given by
    \begin{align*}
        \begin{bmatrix}
            -\lambda & 2\lambda & -\lambda
        \end{bmatrix}
        \begin{bmatrix}
            j^2 \\
            (j + 1) ^ 2\\
            (j + 2) ^ 2\\
        \end{bmatrix}
        &= -\lambda j^2 + 2 \lambda (j + 1)^2 - \lambda (j+2)^2\\
        &= -\lambda j^2 + 2 \lambda (j^2 + 2j + 1) - \lambda (j^2 + 4j + 4)\\
        &= -2 \lambda.
    \end{align*}
    Thus, we have
    \[B \mathbf{x} = \begin{bmatrix}
        4 \lambda (m + 1) + 2\lambda & -2\lambda & \cdots & -2\lambda & -6\lambda & -2\lambda & \cdots & -2\lambda
    \end{bmatrix}^T.\]
    Left-multiplying by $\mathbf{x}^T$ gives us
    \begin{align*}
        \mathbf{x}^T B \mathbf{x} &= \mathbf{x}^T 
        \begin{bmatrix}
            4\lambda(m + 1) + 2\lambda & -2\lambda & \cdots & -2\lambda & -6\lambda & -2\lambda & \cdots & -2\lambda
        \end{bmatrix}^T\\
        &= 4\lambda(m + 1) + 2\lambda - 6\lambda(m + 1)^2 - 2\lambda \left(2\sum_{j = 1}^{m}(j^2) - 2\right)\\
        &= -\frac{4}{3}\lambda m^3 - 8 \lambda m^2 - \frac{26}{3} \lambda m + 4\lambda.
    \end{align*}
Now, adding the two terms together gives us
\begin{align*}
    \mathbf{x}^T \Tilde{\Delta}^{-1} \mathbf{x} &= \frac{1}{4 \lambda m} \left( 2\left(\frac{2m^3 + 6m^2 + 7m}{3}\right) ^ 2 -\frac{4}{3} \lambda m^3 - 8\lambda m^2 -\frac{26}{3}\lambda m + 4 \lambda \right)\\
    &= \frac{1}{4\lambda m}\left(\frac{8}{9}m^6 + \frac{16}{3}m^5 + \frac{128}{9}m^4 + \frac{56}{3}m^3 + \frac{98}{9}m^2 - \frac{4}{3}\lambda m^3 - 8\lambda m^2 -\frac{26}{3}\lambda m + 4\lambda\right).
\end{align*}
Because $\lambda$ is the column sum of $\Tilde{\Delta}$, it is given by $m(2m^2+1)/3$. The previous expression is greater than $0$ for any integer $m$.
    
\end{proof}
%End proof about adding a single pendant

\begin{lemma}\label{lem:cycle_single_pendant}
    Let $U$ be a graph which contains exactly one cycle of length $p$ where $p$ is even and one pendant connected to one of the vertices. Let $\Delta$ be the distance squared matrix of $U$ and $m = p/2$. Then $i(\Delta) = (m, m+1, 0).$
\end{lemma}

\begin{proof}
    The distance squared matrix for $U$ will take the form
    \begin{equation*}
        \begin{bmatrix}
            \tilde{\Delta} & \textbf{x} \\
            \textbf{x}^\top & 0
        \end{bmatrix},
    \end{equation*}
    where $\tilde{\Delta}$ and $\textbf{x}$ take the same definitions as in Lemma \ref{lem:x_delta_inverse_x_positive}. By Theorem \ref{thm:haynsworth_additivity_formula}, Corollary \ref{cor:cycle_eigenvalues}, and Lemma \ref{lem:x_delta_inverse_x_positive}, we'll have that 
    \begin{align*}
        i(\Delta) &= i\left(\begin{bmatrix}
            \tilde{\Delta} & \textbf{x} \\
            \textbf{x}^\top & 0
        \end{bmatrix}\right) \\
        &= i(\tilde{\Delta}) + i(-\textbf{x}^\top \tilde{\Delta}^{-1} \textbf{x}) \\
        &= (m, m, 0) + (0, 1, 0) \\
        &= (m, m+1, 0).
    \end{align*}
\end{proof}

\begin{lemma}\label{lem:one_deg_2}
Let $U$ be a graph which contains exactly one cycle and assume that the cycle contains an even number of vertices. Also assume that of the vertices of the cycle, only one has degree 3 in $U$ and the rest have degree 2 in $U$. If $U$ has at least one vertex of degree 2 not part of the cycle, then $i_0(\Delta) \geq 1$. 
\end{lemma}

\begin{proof}
Let $m$ be the length of the cycle divided by two. Let the vertices in the cycle be numbered from 1 to $2m$. Let $2m+1$ be the number of the vertex connected to $2m$, but not part of the cycle. Let $u$ be a degree two vertex not in the cycle. 

Let $n$ be the number of vertices in $U$. Let the entries of $\mathbf{a} \in \RR^n$ be defined by
\begin{equation*}
    a_j = \begin{cases}
        -1 & \text{if } j = m \\
        m + 1 & \text{if } j = 2m \\
        -m & \text{if } j = 2m + 1 \\
        0 & \text{otherwise}
    \end{cases}
\end{equation*}
and the entries of $\mathbf{b} \in \RR^n$ be defined by
\begin{equation*}
    b_j = \begin{cases}
        m(m + 1)/2 & \text{if } j = u - 1 \text{ or } j = u + 1\\
        -m(m + 1) & \text{if } j = u \\
        0 & \text{otherwise.}
    \end{cases}
\end{equation*}

Let $\Delta$ be the distance squared matrix of $U$ and $\mathbf{r}$ be an arbitrary row in $\Delta$. With the assumptions of the lemma, % Should I replace the beginning of this sentence with assumptions that we used?
observe that
\begin{align*}
    \mathbf{r}^\top \textbf{a} &= (m - z)^2 (-1) + z^2 (m + 1) + (z + 1)^2 (-m) \\
    &= -(m^2 - 2mz + z^2) + z^2m + z^2 - (z^2 + 2z + 1)m \\
    &= -m^2 + 2mz - z^2 + z^2m + z^2 - z^2m - 2mz - m \\
    &= -m^2 - m
\end{align*}
and 
\begin{align*}
    \mathbf{r}^\top \textbf{b} &= \frac{(z + u - 1)^2 m(m + 1)}{2} + (z + u)^2 (-m(m + 1)) + \frac{(z + u + 1)^2 m(m + 1)}{2} \\
    &= \frac{(z^2 + u^2 + 1 + 2zu - 2z - 2u)(m^2 + m)}{2} - (z^2 + 2zu + u^2)(m^2 + m) \\
    &+ \frac{(z^2 + u^2 + 1 + 2zu + 2z + 2u)(m^2 + m)}{2} \\
    &= m^2 + m
\end{align*}
where $z$ corresponds to the location of $\mathbf{r}$ in $\Delta$. Combining these together, we have $\mathbf{r}^\top (\mathbf{a} + \mathbf{b}) = 0$, and so $\Delta(\textbf{a} + \textbf{b})=0$.  Hence, $i_0(\Delta) > 0$ by the Invertible Matrix Theorem. 
\end{proof}

We are now ready for the main result of this section.

\begin{theorem}\label{thm:p_cycle_one_tree}
Let $U$ be a graph on $n > 2$ vertices which has a cycle of $p$ vertices with a tree connected to one of the vertices of the cycle. Assume $p$ is even. Let $\Delta$ be the distance squared matrix on $U$. Then 
\begin{equation*}
    i_-(\Delta) = \ell + q,
\end{equation*}
where $\ell$ is the number of leaves of $U$ and $q$ is the number of negative eigenvalues of the $p$-cycle. 
In addition, if $U$ has $t$ vertices of degree 2 whose removal disconnects the graph then
\[
i_0(\Delta) = t.
\]
\end{theorem}

\begin{proof}
We will use induction on $n$. Note that Lemma \ref{lem:cycle_single_pendant} takes care of the case when when a single leaf is connected to the cycle. So let $U$ be a unicyclic graph on $n \geq 4$ vertices and assume the conclusion is true for all unicyclic graphs on fewer than $n$ vertices. 

Consider a diametrical path in $U$. There exists a vertex $w$ such that $w$ has exactly one non-leaf neighbor $u$.

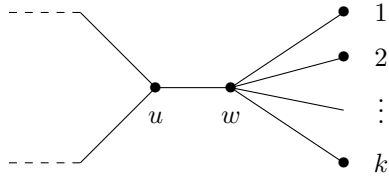
\begin{center}
\begin{tikzpicture}[roundnode/.style={circle, draw=black, fill=white, very thick},scale=2]
		\draw
		(0, 0)node{}--(-0.5, 0.5)node{};
		\draw
		(0, 0)node{}--(-0.5, -0.5)node{};
		\draw[dashed]
		(-0.5, 0.5)node{}--(-1, 0.5)node{};
		\draw[dashed]
		(-0.5, -0.5)node{}--(-1, -0.5)node{};
		\draw
		(0.5, 0)node{}--(1.25, 0.5)node{$\bullet$} (1.5, 0.5)node{1}; 
		\draw
		(0.5, 0)node{}--(1.25, 0.2)node{$\bullet$} (1.5, 0.2)node{2}; 
		\draw
		(0.5, 0)node{}--(1.25, -0.15)node{} (1.5, -0.1)node{\vdots}; 
		\draw
		(0.5, 0)node{}--(1.25, -0.5)node{$\bullet$} (1.5, -0.5)node{$k$}; 
		\draw
		(0, 0)node{$\bullet$}--(0.5, 0)node{$\bullet$} (0, -0.2)node{$u$} (0.5, -0.2)node{$w$}; 
\end{tikzpicture}
\end{center}

Let $w$ be adjacent to $k$ leaves, and label leaves $1, \ldots, k$, $w$ is $k + 1$, and $u$ is $k + 2$. 

\begin{description}
    \item[Case 1:] Suppose $k \geq 2$. 
    %%Then the distance squared matrix is 
    % \begin{equation*}
    %     \Delta = \begin{bmatrix}
    %         4J_k - 4I_k & X \\
    %         X^T & \Delta_{k + 1}
    %     \end{bmatrix},
    % \end{equation*}
    % where $J_n$ is the $n \times n$ matrix of ones, $X$ is a matrix which has the distances from the leaves $1, \ldots, k$ to the rest of the graph, and $\Delta_n$ is the distance squared matrix for $U \setminus {1, \ldots, n - 1}$. Using corresponding row and column operations, $\Delta$ can be written as 
    % \begin{equation*}
    %     \Delta = \begin{bmatrix}
    %         -4I_{k - 1} - 4J_{k - 1} & 0 \\
    %         0 & \cfrac{4(k - 1)}{k} \mathbf{e}_k \mathbf{e}_k^\top + \Delta_k
    %     \end{bmatrix}.
    % \end{equation*}
    % Notice that $-4I_{k - 1} - 4J_{k - 1}$ is negative definite, so it will have $k - 1$ negative eigenvalues. 
    By Proposition \ref{lem:kpend}, $i_-(\Delta) = k-1 + i_-\left(\Delta_{k+1} + \left(4-\frac{4}{k}\right)e_ke_k^T\right)$.
    Let $\frac{4(k - 1)}{k} \mathbf{e}_k \mathbf{e}_k^\top + \Delta_k = F$ and notice $F(1) = \Delta_{k + 1}$. Notice that $U \setminus {1, \ldots, k}$ is a tree with $\ell - (k - 1)$ leaves and at least three vertices. By the inductive hypothesis,
    \begin{equation*}
        i_-(\Delta_{k + 1}) = \ell + q - (k - 1).
    \end{equation*}
    By interlacing, $i_-(F) \geq \ell + q - (k - 1)$. Since $F = \frac{4(k - 1)}{k} \mathbf{e}_k \mathbf{e}_k^\top + \Delta_k$ and $\frac{4(k - 1)}{k} > 0$,
    \begin{align*}
        i_-(F) &\leq i_-\left(\frac{4(k - 1)}{k} \mathbf{e}_k \mathbf{e}_k^\top\right) + i_-(\Delta_k) \\
        &\leq 0 + \ell + q - (k - 1) \\
        &= \ell + q - (k - 1).
    \end{align*}
    Thus, $i_-(F) = \ell + q - (k - 1)$. Now notice 
    \begin{align*}
        i_-(\Delta) &= k-1 + i_-(F) \\
        &= (k - 1) + \ell + q - (k - 1) \\
        &= \ell + q.
    \end{align*}
    
    \item[Case 2:] Suppose $k = 1$ and $w$ is the only degree 2 vertex in the graph. Let $\Delta_2 = \Delta(1)$. By our inductive hypothesis, $i_-(\Delta_2) = \ell + q$. By interlacing, $i_-(\Delta) \geq \ell + q$. Since $p$ is even and $w$ is a degree 2 vertex not in the cycle, we can use Lemma \ref{lem:one_deg_2} to deduce that the column in $\Delta$ corresponding to $w$ is a linear combination of other columns in $\Delta$. Thus, $\text{rank}(\Delta) = \text{rank}(\Delta_2)$ and $i_-(\Delta) \leq \ell + q$. 
    %one of the eigenvalues must be 0. Since we only deleted one row and column in the interlacing, we didn't know the identity of exactly one eigenvalue. Since $i_0(\Delta) > 0$, that unknown eigenvalue must be 0; the number of positive and negative eigenvalues won't change. Thus $i_-(\Delta) = \ell + q$. 
    
    \item[Case 3:] Suppose $k = 1$ and there is some other degree 2 vertex in the graph besides $w$. By an argument similar to the first Case 2 of Theorem \ref{thm:tree_inertia}, we'll have that $i_-(\Delta) = \ell + q$. 
\end{description}

We now proceed to prove that $i_0(\Delta)=t$.  Note that by Lemmas \ref{lem:tree_null}, \ref{lem:tree_lin_ind}, and \ref{lem:deg_2_verts}, we have $t-1$ linearly independent vectors in the nullspace of $\Delta$, and each of the vectors constructed there are zero on all the vertices of the cycle.  Furthermore, from Lemma \ref{lem:one_deg_2}, we have another vector in the nullspace of $\Delta$, and the vector constructed in the proof of Lemma \ref{lem:one_deg_2}, is linearly independent from these other vectors since it is nonzero on one of the vertices of the cycle.  Thus the nullspace of $\Delta$ is at least $t$ dimensional, so $i_0(\Delta)\geq t$.

To show that $i_0(\Delta)\leq t$ we proceed by induction on $n$.  We know the result is true for cycles, and by Lemmas \ref{lem:cycle_single_pendant} and \ref{lem:one_deg_2}, we know the result is true for a cycle with a single pendant attached, or a single path of length 2 attached, so suppose that $G$ has more vertices than either of these. Consider the same $w$ as above, adjacent to $k$ leaves.  We again consider three cases.

\begin{description}
    \item[Case 1:] Suppose $k \geq 2$.  Let $U'$ be the same as $U$ with one of these $k$ leaves deleted and let $\Delta_U$ and $\Delta_{U'}$ denote the distance squared matrices on $U$ and $U'$, respectively.  Note that deleting a leaf from $U$ does not affect the number of degree $2$ vertices in the graph.  We have, by interlacing, $i_+(\Delta_U) \geq i_+(\Delta_{U'})$ and from Theorem \ref{thm:tree_inertia} deleting a leaf from $U$ results in $i_-(\Delta_U) = i_-(\Delta_{U'}) + 1$. Thus rank$(\Delta_U) = i_+(\Delta_U) + i_-(\Delta_U)  = i_+(\Delta_{U'}) + i_-(\Delta_{U'}) + 1$. So the number of zero eigenvalues did not change, and since the number of degree 2 vertices did not change, $i_0(\Delta_U) = t$.

    \item[Case 2:] Suppose $k = 1$. Deleting the pendant vertex does not change the number of leaves of the graph, but decreases the number of degree $2$ vertices by one. Using Lemma \ref{lem:deg_2_verts} and the same reasoning in case 2 of Theorem \ref{thm:tree_inertia}, we have rank$(\Delta_U) = $ rank$(\Delta_U')$. Because interlacing tells us that $i_+(\Delta_U) \geq i_+(\Delta_U')$ and $i_-(\Delta_U) \geq i_- (\Delta_U')$, we have $i_+(\Delta_U) = i_+(\Delta_U')$ and $i_-(\Delta_U) = i_-(\Delta_U')$. Therefore $i_0(\Delta_U) \geq i_0(\Delta_U')$. By the inductive hypothesis $i_0(\Delta_U') = t - 1$, so $i_0(\Delta_U) = t$.
    
    \item[Case 3:] Suppose $k = 2$.  We can call $\Delta_{U''}$ the distance squared matrix for the tree after removing both of its leaves.  Then, by similar reasoning to that in case 3 of Theorem \ref{thm:tree_inertia}, we will have rank$\Delta_{U''} = $ rank$\Delta_U - 2$.  The distance squared matrix of $\Delta_U$ will be identical to that described in Theorem \ref{thm:tree_inertia} except $\Delta_{U''}$ will be different. However, this still leads to the conclusion that adding both of our leaves back onto $\Delta_{U''}$ has no effect on the number of zero eigenvalues of $\Delta_{U}$, so $i_0(\Delta_U) = t$.
    
\end{description}
\end{proof}

\section{Other Families of Unicyclic Graphs}\label{sec:ex}

Theorem \ref{thm:p_cycle_one_tree} generalizes Theorem \ref{thm:tree_inertia} from trees to a particular class of unicyclic graphs.  It is natural to try to generalize further to the class of all unicyclic graphs.  In this section, we will see that the result of Theorem \ref{thm:p_cycle_one_tree} does not hold for all unicyclic graphs, but we will see that some information can be given for certain specific families.

%For general unicyclic graphs we weren't able to get a formula as precise as the last sections, but we were able to get a bound on the number of negative eigenvalues in certain more general situations.

% \begin{theorem}\label{thm:arbitrary_unicyclic_graph}
%     Let $U$ be a graph on $n > 2$ vertices which has a cycle of $p$ vertices. Assume $p$ is even. Assume that each vertex of the cycle has at most three edges connected to it. Let $\Delta$ be the distance squared matrix of $U$,  $\ell$ be the number of leaves of $U$, $q$ be the number of negative eigenvalues of the $p$-cycle, and $\beta$ equal to
%     \begin{equation*}
%         2\left\lfloor{\frac{\lfloor{\sin^{-1}\left(1/\sqrt{p/2}\right)\frac{p}{\pi}}\rfloor + 1}{2}}\right\rfloor
%     .\end{equation*}
%     Then
%     \begin{equation*}
%         \ell + q - \beta \leq i_-(\Delta) \leq \ell + q
%     \end{equation*}
%     if there are no degree 2 vertices in the graph besides the degree 2 vertices of the cycle. Otherwise, 
%     \begin{equation*}
%         \ell + q - \beta \leq i_-(\Delta) \leq \ell + q + 1
%     \end{equation*}
%     if there is at least one degree 2 vertex in the graph besides the degree 2 vertices of the cycle.
% \end{theorem}

% The remainder of this section will be devoted to proving this.

\subsection{Saturated Cycles}\label{sus:saturated_cycles}

Consider the following graph
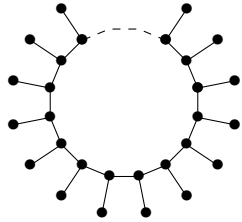
\begin{center}
    \begin{tikzpicture}
        \draw (0.9807852804032304,0.19509032201612825)node{$\bullet$}
        --(0.8314696123025452,0.5555702330196022)node{$\bullet$}
        --(0.5555702330196023,0.8314696123025452)node{$\bullet$};
        \draw[dashed] (0.5555702330196023,0.8314696123025452)node{$\bullet$}
        --(0.19509032201612833,0.9807852804032304)
        --(-0.1950903220161282,0.9807852804032304)
        --(-0.555570233019602,0.8314696123025453)node{$\bullet$};
        \draw (-0.555570233019602,0.8314696123025453)node{$\bullet$}
        --(-0.8314696123025453,0.5555702330196022)node{$\bullet$}
        --(-0.9807852804032304,0.1950903220161286)node{$\bullet$}
        --(-0.9807852804032304,-0.19509032201612836)node{$\bullet$}
        --(-0.8314696123025455,-0.555570233019602)node{$\bullet$}
        --(-0.5555702330196022,-0.8314696123025452)node{$\bullet$}
        --(-0.19509032201612866,-0.9807852804032303)node{$\bullet$}
        --(0.1950903220161283,-0.9807852804032304)node{$\bullet$}
        --(0.5555702330196026,-0.831469612302545)node{$\bullet$}
        --(0.8314696123025452,-0.5555702330196022)node{$\bullet$}
        --(0.9807852804032303,-0.19509032201612872)node{$\bullet$}
        --(0.9807852804032304,0.19509032201612825)node{$\bullet$};
        \draw (0.9807852804032304,0.19509032201612825)--(1.4711779206048456,0.2926354830241924)node{$\bullet$};
        \draw (0.8314696123025452,0.5555702330196022)--(1.2472044184538178,0.8333553495294033)node{$\bullet$};
        \draw (0.5555702330196023,0.8314696123025452)--(0.8333553495294035,1.2472044184538178)node{$\bullet$};
        \draw (-0.555570233019602,0.8314696123025453)--(-0.8333553495294029,1.247204418453818)node{$\bullet$};
        \draw (-0.8314696123025453,0.5555702330196022)--(-1.247204418453818,0.8333553495294033)node{$\bullet$};
        \draw (-0.9807852804032304,0.1950903220161286)--(-1.4711779206048456,0.2926354830241929)node{$\bullet$};
        \draw (-0.9807852804032304,-0.19509032201612836)--(-1.4711779206048456,-0.2926354830241925)node{$\bullet$};
        \draw (-0.8314696123025455,-0.555570233019602)--(-1.2472044184538182,-0.8333553495294029)node{$\bullet$};
        \draw (-0.5555702330196022,-0.8314696123025452)--(-0.8333553495294033,-1.2472044184538178)node{$\bullet$};
        \draw (-0.19509032201612866,-0.9807852804032303)--(-0.292635483024193,-1.4711779206048454)node{$\bullet$};
        \draw (0.1950903220161283,-0.9807852804032304)--(0.29263548302419246,-1.4711779206048456)node{$\bullet$};
        \draw (0.5555702330196026,-0.831469612302545)--(0.8333553495294039,-1.2472044184538176)node{$\bullet$};
        \draw (0.8314696123025452,-0.5555702330196022)--(1.2472044184538178,-0.8333553495294033)node{$\bullet$};
        \draw (0.9807852804032303,-0.19509032201612872)--(1.4711779206048454,-0.29263548302419307)node{$\bullet$};
    \end{tikzpicture}
\end{center}
This graph is a cycle of length $p$ with a pendant extending from each node on the cycle. 
A distanced squared matrix for this graph can take the form 
\begin{equation*}
    H = 
    \begin{bmatrix}
    A & B \\
    B & C
    \end{bmatrix},
\end{equation*}
where $A$, $B$, and $C$ are circulant matrices of size $p \times p$. Here the first $p$ rows and columns are indexed by the nodes of the main cycle and the last $p$ by the pendant vertices.  Since $A$, $B$, and $C$ are all $p \times p$ circulant matrices, $A$, $B$, and $C$ will all have the same eigenvectors by Lemma \ref{lem:eigenstuff_of_circulant_matrices}. Thus, if $\mathbf{x}$ is an eigenvector of $A$ then $A\mathbf{x} = \lambda \mathbf{x}$, $B\mathbf{x} = \mu \mathbf{x}$, and $C\mathbf{x} = \rho \mathbf{x}$, where $\lambda$, $\mu$, and $\rho$ are eigenvalues of $A$, $B$, and $C$ respectively corresponding to the eigenvector $\mathbf{x}$. 

To find the eigenvalues of $H$, we will assume that an eigenvector of $H$ will take the form $\begin{bmatrix}a \mathbf{x} & b \mathbf{x}\end{bmatrix}^T$. Then an eigenvalue $h$ of 
$H$ satisfies
\begin{equation*}
    \begin{bmatrix}
    A & B \\
    B & C
    \end{bmatrix}
    \begin{bmatrix}
    a \mathbf{x} \\
    b \mathbf{x}
    \end{bmatrix}
    =
    h
    \begin{bmatrix}
    a \mathbf{x} \\
    b \mathbf{x}
    \end{bmatrix}
.\end{equation*}
Thus%Using the eigenvalues corresponding to $\mathbf{x}$ above, notice that
\begin{align*}
    % \begin{bmatrix}
    % A & B \\
    % B & C
    % \end{bmatrix}
    % \begin{bmatrix}
    % a \mathbf{x} \\
    % b \mathbf{x}
    % \end{bmatrix}
    % &=
    % h
    % \begin{bmatrix}
    % a \mathbf{x} \\
    % b \mathbf{x}
    % \end{bmatrix}
    % \\
    % \begin{bmatrix}
    % A & B \\
    % B & C
    % \end{bmatrix}
    % \begin{bmatrix}
    % a \mathbf{x} \\
    % b \mathbf{x}
    % \end{bmatrix}
    % &=
    % \begin{bmatrix}
    % ha \mathbf{x} \\
    % hb \mathbf{x}
    % \end{bmatrix}
    % \\
    % \begin{bmatrix}
    % Aa\mathbf{x} +  Bb\mathbf{x} \\
    % Ba\mathbf{x} +  Cb\mathbf{x}
    % \end{bmatrix}
    % &=
    % \begin{bmatrix}
    % ha \mathbf{x} \\
    % hb \mathbf{x}
    % \end{bmatrix}
    % \\
    % \begin{bmatrix}
    % a\lambda\mathbf{x} +  b\mu\mathbf{x} \\
    % a\mu\mathbf{x} +  b\rho\mathbf{x}
    % \end{bmatrix}
    % &=
    % \begin{bmatrix}
    % ha \mathbf{x} \\
    % hb \mathbf{x}
    % \end{bmatrix}
    % \\
    \begin{bmatrix}
    (a\lambda +  b\mu)\mathbf{x} \\
    (a\mu +  b\rho)\mathbf{x}
    \end{bmatrix}
    &=
    \begin{bmatrix}
    (ha) \mathbf{x} \\
    (hb) \mathbf{x}
    \end{bmatrix}
.\end{align*}
%So $a\lambda + b\mu = ha$ and $a\mu + b\rho = hb$. 
Solving for $h$ gives $h = (a\lambda + b\mu)/a$ and $h = (a\mu + b\rho)/b$. %Setting the right hand side of each equation equal to each other with some algebra 
This gives
\begin{align*}
    % \frac{a\lambda + b\mu}{a} &= \frac{a\mu + b\rho}{b} \\
    % ab\lambda + b^2\mu &= a^2\mu + ab\rho \\
    % b^2\mu - a^2\mu &= ab\rho - ab\lambda \\
    % \frac{b^2 - a^2}{ab} &= \frac{\rho - \lambda}{\mu} \\
    \frac{b}{a} - \frac{a}{b} &= \frac{\rho - \lambda}{\mu}
.\end{align*}
%Now the eigenvalue $h$ will have many multiples of $\begin{bmatrix}a \mathbf{x} & b \mathbf{x}\end{bmatrix}^T$ as eigenvectors. 
Set $b = 1$, then solve for $a$ to obtain
\begin{align*}
    % \frac{b}{a} - \frac{a}{b} &= \frac{\rho - \lambda}{\mu} \\
    % \frac{1}{a} - a &= \frac{\rho - \lambda}{\mu} \\
    1 - a^2 &= a\left(\frac{\rho - \lambda}{\mu}\right) %\\
    % a^2 + a\left(\frac{\rho - \lambda}{\mu}\right) -1 &= 0.
\end{align*}
which yields
\begin{equation*}
    a = \frac{-\left(\frac{\rho - \lambda}{\mu}\right) \pm \sqrt{\left(\frac{\rho - \lambda}{\mu}\right)^2 + 4}}{2}.
\end{equation*}
Then we can solve for $h$ using $a\mu + b\rho = hb$. Since $b = 1$, we have $a\mu + \rho = h$. Hence,
% \begin{align*} 
%     h &= a\mu + \rho \\
%     &= \left(\frac{-\left(\frac{\rho - \lambda}{\mu}\right) \pm \sqrt{\left(\frac{\rho - \lambda}{\mu}\right)^2 + 4}}{2}\right)\mu + \rho \\
%     &= -\frac{(\rho - \lambda)}{2} \pm \frac{\mu}{2}\sqrt{\left(\frac{\rho - \lambda}{\mu}\right)^2 + 4} + \rho
% \end{align*}
    %&= \frac{\rho}{2} + \frac{\lambda}{2} \pm \sqrt{\frac{(\rho - \lambda)^2}{4} + \mu^2}
\begin{equation} \label{eq:1}
    h = \frac{\rho}{2} + \frac{\lambda}{2} \pm \sqrt{\frac{(\rho - \lambda)^2}{4} + \mu^2}
.\end{equation}
So the eigenvalues of $H$ will all correspond to values of the expression above for each combination of $\lambda$, $\mu$, $\rho$ corresponding to $\textbf{x}$.

Now we need to find expressions for $\lambda$, $\mu$, and $\rho$, which depend on $A$, $B$, and $C$. Since these matrices are circulant, then by Lemma \ref{lem:eigenstuff_of_circulant_matrices}, the eigenvalues are of the form $c_1 + c_2\omega_j + \cdots + c_p\omega_j^{p - 1}$ where $c_1, \ldots, c_p$ are the entries of the first row and $\omega = e^{2\pi i/p}$. 

We are going to consider an even cycle now. The first row of $A$ will take the form 
\begin{equation*}
    \begin{bmatrix}0 & 1 & 4 & \cdots & (m - 1)^2 & m^2 & (m - 1)^2 & \cdots & 4 & 1\end{bmatrix},
\end{equation*}
where $m = p/2$.\\
So for every index $j \in \{0, \ldots, p - 1\}$, we will have, omitting some tedious calculations, %  XD
{\allowdisplaybreaks
\begin{align*}
\begin{split}
    \lambda_j &= \omega^j + 4\omega^{2j} + \cdots + (m - 1)^2\omega^{(m - 1)j} + m^2\omega^{mj} + (m - 1)^2\omega^{(p - m + 1)j} + \cdots \\
    &\qquad + 4\omega^{(p - 2)j} + \omega^{(p - 1)j} \\
\end{split}\\
    &= 2\sum_{\ell = 1}^{m - 1} \ell^2 \cos\left(\frac{2\pi}{p} \ell j\right) + m^2\cos\left(\pi j\right).
\end{align*}
}
The first row in $B$ will take the form
\begin{equation*}
    \begin{bmatrix}
        1 & 4 & 9 & \cdots & m^2 & (m + 1)^2 & m^2 & \cdots & 9 & 4
    \end{bmatrix}.
\end{equation*}
We can do a similar calculation above for every $j \in \{0, \ldots, p - 1\}$ to find
\begin{equation*}
    \mu_j = 2\sum_{\ell = 1}^{m - 1} (\ell + 1)^2 \cos\left(\frac{2\pi}{p} \ell j\right) + (m + 1)^2\cos\left(\pi j\right) + 1.
\end{equation*}
The first row in $C$ will take the form
\begin{equation*}
    \begin{bmatrix}
        0 & 9 & 16 & \cdots & (m + 1)^2 & (m + 2)^2 & (m + 1)^2 & \cdots & 16 & 9
    \end{bmatrix}.
\end{equation*}
Again performing similar calculations for each $j \in \{0, \ldots, p - 1\}$, we find
\begin{equation*}
    \rho_j = 2\sum_{\ell = 1}^{m - 1} (\ell + 2)^2 \cos\left(\frac{2\pi}{p} \ell j\right) + (m + 2)^2\cos\left(\pi j\right).
\end{equation*}
Using algebraic and trigonometric manipulations, we can find that 
\begin{align*}
    \lambda_j &= 2\sum_{\ell = 1}^{m - 1} \ell^2 \cos\left(\frac{2\pi}{p} \ell j\right) + m^2\cos\left(\pi j\right) \\
    &= m\cos(\pi j)\left(-m + \csc\left(\frac{\pi}{2m}j\right)^2\right) + \frac{1}{4} \left(4m^2\cot\left(\frac{\pi}{2m}j\right) - \csc\left(\frac{\pi}{2m}j\right)^4\sin\left(\frac{\pi}{2m}j\right)\right)\sin\left(\pi j\right) \\ 
    &~~~+ m^2\cos(\pi j).
\end{align*}
Since $j$ is an integer, the $\sin(\pi j)$ will always evaluate to 0. Therefore, we have
\begin{align*}
    \lambda_j &= m\cos(\pi j)\left(-m + \csc\left(\frac{\pi}{2m}j\right)^2\right) + m^2\cos(\pi j) \\
    &= m\cos(\pi j)\csc\left(\frac{\pi}{2m}j\right)^2 \\
    &= (-1)^j m \csc\left(\frac{\pi}{2m}j\right)^2.
\end{align*}

Using similar manipulations and the fact that $j$ is an integer, we can find that
\begin{equation*}
    \mu_j = (-1 + (1 + m)(-1)^j)\csc\left(\frac{\pi}{2m}j\right)^2
\end{equation*}
and
\begin{equation*}
    \rho_j = -4 + (-2 + (2 + m)(-1)^j)\csc\left(\frac{\pi}{2m}j\right)^2.
\end{equation*}
%We will now determine the amount of eigenvalues in the saturated cycle with the $\lambda_j$'s, $\mu_j$'s, and $\rho_j$'s.

Assume $j$ is even and not 0. Then
\begin{equation*}
    \lambda_j = m\csc\left(\frac{\pi j}{2m}\right)^2,
\end{equation*}
\begin{equation*}
    \mu_j = m\csc\left(\frac{\pi j}{2m}\right)^2,
\end{equation*}
and
\begin{equation*}
    \rho_j = -4 + m\csc\left(\frac{\pi j}{2m}\right)^2
.\end{equation*}
Now that we have expressions for the eigenvalues of $A$, $B$, $C$, we will determine the sign of the eigenvalues of the larger matrix $H$. Equation $(\ref{eq:1})$ gives the expression for these eigenvalues. Note that it contains a plus-or-minus sign and that for the positive case of (\ref{eq:1}), the output of the square root will always be positive. Adding $\rho$ and $\lambda$, we get
\begin{equation*}
    -4+2m\csc\left(\frac{\pi j}{2m}\right)^{2}.
\end{equation*}
Observe that the cosecant squared term will have a minimum value of 1. Since $m > 1$, this expression must be nonnegative. Therefore, when $j$ is even, and we are concerned with the positive side case of (\ref{eq:1}), we will only get positive eigenvalues from (\ref{eq:1}). 

Now, for the negative case of (\ref{eq:1}), we will show that
\[\frac{\rho}{2} + \frac{\lambda}{2} - \sqrt{\frac{(\rho - \lambda)^2}{4} + \mu^2} < 0.\]
Manipulating this, we find that

    % \frac{\rho}{2} + \frac{\lambda}{2} &< \sqrt{\frac{(\rho - \lambda)^2}{4} + \mu^2} \\
    % \frac{\rho^2 + 2\rho\lambda + \lambda^2}{4} &< \frac{\rho^2 - 2\rho\lambda + \lambda^2}{4} + \frac{4\mu^2}{4} \\
    \[\rho \lambda < \mu^2.\]
Using the definitions of $\lambda$, $\mu$, and $\rho$ above and letting $a = \csc\left(\frac{\pi j}{2m}\right)$, notice that
\begin{equation*}
    \rho \lambda = -4ma^2 + m^2a^2 < m^2a^2 = \mu^2.
\end{equation*}
So all the eigenvalues that are produced under those conditions will be negative. This corresponds to the number of negative eigenvalues of the cycle or $q$ in the above theorem. 

When $j$ is zero, $\ell$ will always be positive, so the summations above will evaluate to be positive.

Now assume $j$ is odd. Then
\begin{equation*}
    \lambda_j = -m \csc\left(\frac{\pi}{2m}j\right)^2,
\end{equation*}
\begin{equation*}
    \mu_j = (-2 - m) \csc\left(\frac{\pi}{2m}j\right)^2,
\end{equation*}
and
\begin{equation*}
    \rho_j = -4 + (-4 - m) \csc\left(\frac{\pi}{2m}j\right)^2
.\end{equation*}
So then, once again omitting some calculations,
\small
\begin{align*}
    \frac{\rho_j}{2} + \frac{\lambda_j}{2} \pm \sqrt{\frac{(\rho_j - \lambda_j)^2}{4} + \mu_j^2} &= -2 - (2 + m) \csc\left(\frac{\pi}{2m}j\right)^2 \\
    % % &= \frac{1}{2}\left(-4 + (-4 - m) \csc\left(\frac{\pi}{2m}j\right)^2 - m \csc\left(\frac{\pi}{2m}j\right)^2\right) \\ &\pm \sqrt{\frac{1}{4}\left(-4 + (-4 - m) \csc\left(\frac{\pi}{2m}j\right)^2 + m \csc\left(\frac{\pi}{2m}j\right)^2\right)^2 + \left((-2 - m) \csc\left(\frac{\pi}{2m}j\right)^2\right)^2} \\
    % &= -2 - (2 + m) \csc\left(\frac{\pi}{2m}j\right)^2 \pm \sqrt{\frac{1}{4}\left(-4 + -4 \csc\left(\frac{\pi}{2m}j\right)^2\right)^2 + (2 + m)^2 \csc\left(\frac{\pi}{2m}j\right)^4} \\
    &\pm \sqrt{4 + 8\csc\left(\frac{\pi}{2m}j\right)^2 + (m^2 + 4m + 8) \csc\left(\frac{\pi}{2m}j\right)^4}
.\end{align*}
\normalsize
Essentially, this equation will give pairs of eigenvalues for the saturated graph. Notice that for the negative case, the square root function is defined to give positive values, $m$ is positive, and the cosecant squared term will also be positive, so the above expression must be negative no matter which value of $j$ we choose. 

For the positive case, we need to set the equation to zero and find where the roots are so we can see which values of $j$ produce positive or negative eigenvalues. We will be using the continuity of the sine function and the $1/x^2$ function in the range $(0, p)$. Setting the equation to 0 and simplifying gives
\begin{equation*}
    \frac{1}{\sin\left(\frac{\pi}{2m} j\right)^2} = m
\end{equation*}
or 
\begin{equation*}
    \frac{1}{\sin\left(\frac{\pi}{2m} j\right)^2} = 0
.\end{equation*}
The second equation is impossible, so we will proceed with solving for $j$ in the first equation. This gives the roots
\begin{equation*}
    \sin^{-1}\left(\frac{1}{\sqrt{m}}\right)\frac{2m}{\pi}
\end{equation*}
for a given $m$. The inverse sine can output multiple values, but we are only concerned with the range from 0 to $p$. When we apply this restriction, this expression will give two roots in the range from 0 to $p$. 

%We will now show...

% TODO: Make it so these variables are not at the first of sentences.
Since a cycle needs at least 3 vertices, $m$ must be greater than $1$. When we substitute $m$ for $j$ in the equation above, we get $-4 - m + \sqrt{m^2 + 4m + 20}$. By inspection, we can see that $m$ will always be between the two roots. We will now show that this resulting expression is always negative. We need to show $0 > -4 - m + \sqrt{m^2 + 4m + 20}$. This is equivalent to showing
\begin{equation*}
    m^2 + 8m + 16 > m^2 + 4m + 20.
\end{equation*}
Since $m > 1$, this expression is true. Therefore, by continuity we'll have that all the eigenvalues produced between the two roots will be negative and the eigenvalues produced above and below the roots, but within the range $(0, p)$ will be positive.  

Now note that if all the eigenvalues were negative in the case that $j$ is odd, the number of those negative eigenvalues would correspond with the $\ell$ term in the above theorem. However as shown above, there are a few eigenvalues that are positive in this case. So let $r$ be the lower root. When we apply the function
\begin{equation*}
   \beta = 2\left\lfloor{\frac{\lfloor{r}\rfloor + 1}{2}}\right\rfloor
\end{equation*}
this will give us the number of positive eigenvalues that we need to subtract off to get the lower bound, $\ell + q - \beta$, of the theorem. Overall in the case when the cycle is saturated in the unicyclic graph, we will have 
\begin{equation*}
    i_-(\Delta) = \ell + q - \beta
.\end{equation*}

\subsection{Cycles with a fixed number of pendants per vertex}\label{subsec:manypendant}

% \begin{theorem}\label{thm:more_arbitrary_unicyclic_graph}
%     Let $U$ be a graph on $n > 2$ vertices which has a cycle of $p$ vertices. Assume $p$ is even. Let $\Delta$ be the distance squared matrix of $U$,  $\ell$ be the number of leaves of $U$, $q$ be the number of negative eigenvalues of the $p$-cycle.
%     Then
%     \begin{equation*}
%         \ell \leq i_-(\Delta) \leq \ell + q
%     \end{equation*}
%     if there are no degree 2 vertices in the graph besides the degree 2 vertices of the cycle. Otherwise, 
%     \begin{equation*}
%         \ell \leq i_-(\Delta) \leq \ell + q + 1
%     \end{equation*}
%     if there is at least one degree 2 vertex in the graph besides the degree 2 vertices of the cycle.
% \end{theorem}

In this section we will consider unicyclic graphs consisting of a single even cycle and a fixed number of pendants at least the length of the cycle divided by two coming off of each of its vertices.

Here is an example with a graph that contains a cycle of length six.

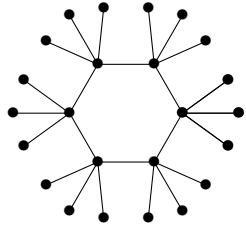
\begin{center}
    \begin{tikzpicture}
        \draw
(0.750, 0.000)node{$\bullet$}--(0.375, 0.650)node{$\bullet$}--(-0.375, 0.650)node{$\bullet$}--(-0.750, 0.000)node{$\bullet$}--(-0.375, -0.650)node{$\bullet$}--(0.375, -0.650)node{$\bullet$}--(0.750, -0.000)node{$\bullet$};
\draw
(0.750, 0.000)--(1.357, -0.441)node{$\bullet$};
\draw
(0.750, 0.000)--(1.500, 0.000)node{$\bullet$};
\draw
(0.750, 0.000)--(1.357, 0.441)node{$\bullet$};
\draw
(0.375, 0.650)--(1.060, 0.955)node{$\bullet$};
\draw
(0.375, 0.650)--(0.750, 1.299)node{$\bullet$};
\draw
(0.375, 0.650)--(0.297, 1.395)node{$\bullet$};
\draw
(-0.375, 0.650)--(-0.297, 1.395)node{$\bullet$};
\draw
(-0.375, 0.650)--(-0.750, 1.299)node{$\bullet$};
\draw
(-0.375, 0.650)--(-1.060, 0.955)node{$\bullet$};
\draw
(-0.750, 0.000)--(-1.357, 0.441)node{$\bullet$};
\draw
(-0.750, 0.000)--(-1.500, 0.000)node{$\bullet$};
\draw
(-0.750, 0.000)--(-1.357, -0.441)node{$\bullet$};
\draw
(-0.375, -0.650)--(-1.060, -0.955)node{$\bullet$};
\draw
(-0.375, -0.650)--(-0.750, -1.299)node{$\bullet$};
\draw
(-0.375, -0.650)--(-0.297, -1.395)node{$\bullet$};
\draw
(0.375, -0.650)--(0.297, -1.395)node{$\bullet$};
\draw
(0.375, -0.650)--(0.750, -1.299)node{$\bullet$};
\draw
(0.375, -0.650)--(1.060, -0.955)node{$\bullet$};
\draw
(0.750, -0.000)--(1.357, -0.441)node{$\bullet$};
\draw
(0.750, -0.000)--(1.500, -0.000)node{$\bullet$};
\draw
(0.750, -0.000)--(1.357, 0.441)node{$\bullet$};
    \end{tikzpicture}
\end{center}

\begin{theorem}
Let $U$ be a unicyclic graph as pictured above with cycle of length $2q$ where each vertex of the cycle has at least $s\geq q$ pendants attached to it, and $\ell$ total pendants.  Then $i_-(\Delta)=\ell$.
\end{theorem}

\begin{proof}
Note that more than one pendant is attached to each vertex of the cycle. A distance squared matrix of this graph can take the form
\begin{equation*}
    H = \begin{bmatrix}
        A & B & X_{13} & X_{14} & \cdots & X_{1,k+1} \\
        B & C & X_{23} & X_{24} & \cdots & X_{2,k+1} \\
        X_{13} & X_{23} & 4J - 4I & X_{34} & \cdots & X_{3,k+1} \\
        X_{14} & X_{24} & X_{34} & 4J - 4I & \cdots & X_{4,k+1} \\
        \vdots & \vdots & \vdots & \vdots & \ddots & \vdots \\
        X_{1,k+1} & X_{2,k+1} & X_{3,k+1} & X_{4,k+1} & \cdots & 4J - 4I
    \end{bmatrix}
\end{equation*}
where $A$, $B$, and $C$ have the same definitions as above and $s$ is the number of pendants attached to each vertex of the cycle. Now we will deal with the matrices which take the form $X_{ij}$. When $i = 1$, $X_{ij}$ is indexed by the connections from the vertices of the cycle and one set of nodes that are connected to a vertex of the cycle except for one node. When $i = 2$, $X_{ij}$ is indexed by the connections from the vertices of a set of $p$ nodes which contains one node from each group of nodes connected to a vertex of the cycle and one set of nodes that are connected to a vertex of the cycle except for one node. The rest of the $X_{ij}$ matrices are of size $k-1 \times k-1$ which contain all the same number and are indexed by groups of nodes (besides one) connected to one vertex of the cycle to other groups of nodes (besides one) connected to another vertex of the cycle. 

Using equivalent row and column operations, we can find the following congruent matrix to $H$. 
\begin{equation*}
    H' = \begin{bmatrix}
        A & B & 0 & 0 & \cdots & 0 \\
        B & C' & 0 & 0 & \cdots & 0 \\
        0 & 0 & -4I - 4J & 0 & \cdots & 0 \\
        0 & 0 & 0 & -4I - 4J & \cdots & 0 \\
        \vdots & \vdots & \vdots & \vdots & \ddots & \vdots \\
        0 & 0 & 0 & 0 & \cdots & -4I - 4J
    \end{bmatrix}
\end{equation*}
where $C' = C + (4 - 4/k)I$. 

Since $-4I - 4J$ is negative definite, all of its eigenvalues will be negative. The $-4I - 4J$ groups will correspond to the all the pendants on the cycle except for one set. The proof in similar to what is in section \ref{sus:saturated_cycles}, except that $C$ has been slightly modified. The definitions of $\lambda$ and $\rho$ will remain the same as the above sections. With $C'$, though, we'll have 
\begin{equation*}
    \rho = -4 + m \csc\left(\frac{\pi j}{2m}\right)^2 + 4 - \frac 4 k = m \csc\left(\frac{\pi j}{2m}\right)^2 - \frac 4 k
\end{equation*}
when $j$ is even and not 0 and 
\begin{equation*}
    \rho = -4 + (-4 - m) \csc\left(\frac{\pi j}{2m}\right)^2 + 4 - \frac 4 k = (-4 - m) \csc\left(\frac{\pi j}{2m}\right)^2 - \frac 4 k
\end{equation*}
when $j$ is odd.

Now we will determine the signs of the rest of the eigenvalues of $H$. When the index $j$ is even, it will be similar to when $j$ is even in section \ref{sus:saturated_cycles}. However, things change when $j$ is odd. Using this new definition of $\rho$ which corresponds with $j$ being odd in positive case of equation (\ref{eq:1}), gives us
\begin{equation*}
    \frac{1}{2}\left(-\frac{4}{k}-2\left(2+m\right)\csc\left(\frac{\pi j}{2m}\right)^{2}+\sqrt{\frac{16}{k^{2}}+\frac{32}{k}\csc\left(\frac{\pi j}{2m}\right)^{2}+4\left(8+m\left(4+m\right)\right)\csc\left(\frac{\pi j}{2m}\right)^{4}}\right).
\end{equation*}
Setting this expression equal to zero and solving for $k$ gives
\begin{equation*}
    k = m\sin\left(\frac{\pi j}{2m}\right)^2
\end{equation*}
Thus, $k$ needs to be at least $m$ to ensure that none of the eigenvalues that come from this will be negative. 

Since $\rho$ and $\lambda$ are negative, the result will be negative in the negative case of equation (\ref{eq:1}). 

Thus, half the eigenvalues returned from (\ref{eq:1}) will be negative and the other half will be non-negative when $j$ is even and when $j$ is odd. Thus,

\begin{equation*}
    i_-(\Delta) = \ell. 
\end{equation*}

\end{proof}

Using various techniques touched on throughout this paper, it should be possible to show that no matter however many trees are attached to the pendants of the particular graph from this section, the number of negative eigenvalues will correspond to its leaves. 

\subsection{Even cycles with opposite pendants}\label{subsec:opp}

Let $T$ be a graph on $n > 2$ vertices which has a cycle of $2k$ vertices and exactly two pendants. One will be opposite of the other. Let $\Delta$ be the distance squared matrix of $T$. Let the ordering of the vertices go from 0 to 2k around the cycle, then connect the next vertex to the 0th vertex and the next vertex to the $k$th vertex. Let $x$, $y$, $z$, and $z$ be the columns of $\Delta$ which correspond to the 0th, $k$th, $(2k)$th, and $(2k+1)$th columns respectively. Then we will have
\begin{equation*}
    \z{x}
    =
    \begin{bmatrix}
    0 \\ 1 \\ \vdots \\ (k - 1)^2 \\ k^2 \\ (k - 1)^2 \\ \vdots \\ 4 \\ 1 \\ 1 \\ (k + 1)^2
    \end{bmatrix}
    ,\quad \quad
    \z{y} = 
    \begin{bmatrix}
    k^2 \\ (k - 1)^2 \\ \vdots \\ 1 \\ 0 \\ 1 \\ \vdots \\ (k - 2)^2 \\ (k - 1)^2 \\ (k + 1)^2 \\ 1
    \end{bmatrix}
    ,\quad \quad
    \z{z} = 
    \begin{bmatrix}
    1 \\ 4 \\ \vdots \\ k^2 \\ (k + 1)^2 \\ k^2 \\ \vdots \\ 9 \\ 4 \\ 0 \\ (k + 2)^2
    \end{bmatrix}
    ,\quad \quad
    \z{w} = 
    \begin{bmatrix}
    (k + 1)^2 \\ k \\ \vdots \\ 4 \\ 1 \\ 4 \\ \vdots \\ (k - 1)^2 \\ k^2 \\ (k + 2)^2 \\ 0
    \end{bmatrix}
\end{equation*}
Generalize this to $k - i$ for the $n - 2$ vertices. 

We want to show $\z{x} - \z{y} - \frac{k}{k + 2}(\z{z} - \z{w}) = 0$. Notice that when we take the first entry in these vectors and perform the operation, we get
\begin{align*}
    0 - k^2 - \frac{k}{k + 2}(1 - (k + 1)^2) 
    %&= -k^2 - \frac{k}{k + 2}(1 - k^2 - 2k - 1) \\
    &= -k^2 - \frac{k}{k + 2}(-k^2 - 2k) \\
    %&= -k^2 + \frac{k}{k + 2}(k^2 + 2k) \\
    &= -k^2 + \frac{k}{k + 2}(k(k + 2)) \\
    &= -k^2 + k^2 \\
    &= 0
\end{align*}
We can do a similar operation to all other entries. Thus it is possible to find a linear combination of three of the vectors above to get the other. Therefore, $\Delta$ will have at least one 0 eigenvalue.

\subsection{Triangle with one tree}\label{Triangle}

\begin{theorem}\label{thm:one_tree_cycle}
Let $U$ be a graph on $n > 2$ vertices which has a cycle of 3 vertices and a tree connected to one of the vertices of the cycle. Let $\Delta$ be the distance squared matrix of $U$. Then
\begin{equation*}
    i_-(\Delta) = \ell + 2,
    \quad \quad \quad
\end{equation*}
where $\ell$ is the number of leaves of $U$.
\end{theorem}

\begin{proof}
Use induction on $n$. For $n = 3$, we know from Corollary \ref{cor:cycle_eigenvalues} that there will be 2 negative eigenvalues of $\Delta$ which satisfies the conclusion of the theorem. 

When $n = 4$, we have a cycle of 3 vertices, one of which is connected to exactly one other vertex not contained in the cycle. The distance squared matrix of this graph is then

\begin{equation*}
    \Delta = \begin{bmatrix}
        0 & 1 & 4 & 4 \\
        1 & 0 & 1 & 1 \\
        4 & 1 & 0 & 1 \\
        4 & 1 & 1 & 0 \\
    \end{bmatrix},
\end{equation*}
resulting in $i_-(\Delta) = 3 = \ell + 2$. %We now look at unicyclic graphs of $n > 4$ vertices.

Let $U$ be a unicyclic graph on $n > 4$ and observe diam($T$)$\geq2$.
%and assume the conclusion is true for all trees on fewer than $n$ vertices (but at least 3 vertices). 
If diam($T$) = 2, then
\begin{equation*}
    \Delta = \begin{bmatrix}
        4J - 4I & \textbf{1} & 4\textbf{1} & 4\textbf{1} \\
        \textbf{1}^T & 0 & 1 & 1 \\
        \textbf{4}^T & 1 & 0 & 1 \\
        \textbf{4}^T & 1 & 1 & 0 \\
    \end{bmatrix},
\end{equation*}
where $4J - 4I$ is $(n - 3) \times (n - 3)$.
We will have that $i_-(4J - 4I) = n - 4$ ($-4$ with a multiplicity of $n - 4$) and $i_+(4J - 4I) = 1$ ($4(n - 2)$). $\Delta$ is equivalent (performing similar row and column operations) to 
\begin{equation*}
    \begin{bmatrix}
        4J - 4I & \textbf{0} & \textbf{0} & \textbf{0} \\
        \textbf{0}^T &  -\cfrac{n - 3}{4(n - 4)} & 1 & 1 \\
        \textbf{0}^T & 1 & -8 & -7 \\
        \textbf{0}^T & 1 & -7 & -8
    \end{bmatrix}.
\end{equation*}

Continuing to perform similar row and column operations, we can diagonalize the upper left $3 \times 3$ matrix, giving us
\[\begin{bmatrix}
    4J - 4I & \textbf{0} & \textbf{0} & \textbf{0} \\
    \textbf{0}^T & -\cfrac{7n - 13}{60(n - 4)} & 0 & 0\\
    \textbf{0}^T & 0 & -15/8 & 0\\
    \textbf{0}^T & 0 & 0 & -8\\
\end{bmatrix}.\]
As $n > 4$, the lower right $3 \times 3$ matrix only has negative eigenvalues, so we have that $i_-(\Delta) = 3 + (n - 4) = n - 1 = \ell + 2$ and $i_0(\Delta) = 0$. 

Thus from here we assume $\text{diam}(U) \geq 3$. Consider a diametrical path in $U$. There exists a vertex $w$ such that $w$ has exactly one non-leaf neighbor $u$. Let $w$ be adjacent to $k$ leaves, and label leaves $1, \ldots, k$, $w$ is $k + 1$, and $u$ is $k + 2$. 

\textbf{Case 1:} Assume $k \geq 2$. %Looking at the distance squared matrix of $U$, we see that the columns in the lower left $(n - k) \times k$ submatrix are all identical, so using row and column $k$, we can zero out the top right and lower left corners (like in section 1). Continuing performing similar row and column operations, we see that $\Delta$ is equivalent to 
By the proof of Proposition \ref{lem:kpend},
\begin{equation*}
    A = \begin{bmatrix}
        -4I_{k - 1} - 4 J_{k - 1} & 0 \\
        \\
        0 & \left(4 - 4/k\right)e_ke_k^T + \Delta_k
    \end{bmatrix},
\end{equation*} 
where $\Delta_k$ is the square distance matrix for $U \setminus \{1, \ldots, k - 1\}$. Notice that $-4I_{k - 1} - 4 J_{k - 1}$ is negative definite and so has $k - 1$ negative eigenvalues. Let $\left(4 - 4/k\right)e_ke_k^T + \Delta_k = B$ and notice $B(1) = \Delta_{k + 1}$ where $\Delta_{k + 1}$ is the distance squared matrix for $U \setminus \{1, \ldots, k\}$. Notice that $U \setminus \{1, \ldots, k\}$ is a unicyclic graph with $\ell - (k - 1)$ leaves and at least 4 vertices. By our inductive hypothesis, we have 
$i_-(\Delta_{k + 1}) = \ell - (k - 1) + 2.$
By interlacing, 
$i_-(B) \geq \ell - (k - 1) + 2.$
Since $B = \left(4 - \frac{4}{k}\right)e_ke_k^T + \Delta_k$ and $4 - \frac{4}{k} > 0$, we have 
\begin{align*}
    i_-(B) &\leq i_-\left(\left(4 - \frac{4}{k}\right)e_ke_k^T\right) + i_-(\Delta_k) \\
    &= 0 + \ell - (k - 1) + 2 \\
    &= \ell - (k - 1) + 2.
\end{align*}
(We used the inductive hypothesis on $U \setminus \{1, \ldots, k - 1\}$ to find $i_-(\Delta_k)$.) Thus $i_-(B) = \ell - (k - 1) + 2$. Furthermore, 
\begin{align*}
    i_-(A) &= i_-(-4I_{k - 1} - 4J_{k - 1}) + i_-(B) \\
    &= k - 1 + \ell - (k - 1) + 2 \\
    &= \ell + 2.
\end{align*}

\textbf{Case 2:} Assume $k = 1$. Because diam$(U) \geq 3$, we know that there exists some vertex $y$ that is adjacent to a vertex in the cycle (call this vertex $x$) and has at least one other neighbor $z$. 
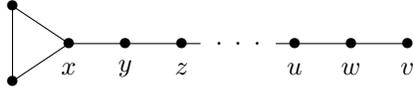
\begin{center}
\begin{tikzpicture}

\draw (0,0)--(-3/4,-1/2)node{$\bullet$}--(-3/4,1/2)node{$\bullet$}--(0,0)node{$\bullet$}--(3/4,0)node{$\bullet$}--(6/4,0)node{$\bullet$}--(7/4,0) (9/4,0)node{. . .} (11/4,0)--(12/4,0)node{$\bullet$}--(15/4,0)node{$\bullet$}--(18/4,0)node{$\bullet$} (0,-1/3)node{$x$} (3/4,-1/3)node{$y$} (6/4,-1/3)node{$z$} (12/4,-1/3)node{$u$} (15/4,-1/3)node{$w$} (18/4,-1/3)node{$v$};
    
\end{tikzpicture}
\end{center}
We again consider two distinct cases.

\textbf{Case 2a:} Suppose $x \neq u$. Then diam$(U) \geq 4$. We let $\Delta_2 = \Delta(1)$. Now using our inductive hypothesis, we have $i_-(\Delta_2) = \ell + 2$. So, by interlacing, we get $i_-(\Delta) \geq \ell + 2$. Then, once again making use of Lemma \ref{lem:tree_null}, we see that the columns of $\Delta$ corresponding to $x, y, z$ with weights $1, -2, 1$ form the all $2$s vector. Thus rank$\Delta = \text{rank}\Delta_2$ and, therefore, $i_-(\Delta) = \ell + 2$.

\textbf{Case 2b:} Suppose diam$(U) = 3$, with $u = x$ a part of the cycle, $w$ its neighbor, and $v$ the only leaf adjacent to $w$.

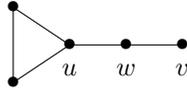
\begin{center}
\begin{tikzpicture}

\draw (0,0)--(-3/4,-1/2)node{$\bullet$}--(-3/4,1/2)node{$\bullet$}--(0,0)node{$\bullet$}--(3/4,0)node{$\bullet$}--(6/4,0)node{$\bullet$}(0,-1/3)node{$u$} (3/4,-1/3)node{$w$} (6/4,-1/3)node{$v$};

\end{tikzpicture}
\end{center}

By direct computation, we have that $i_-(\Delta) = 3 = \ell + 2$ as desired. 

\end{proof}

\section{Conclusion and Further Research}\label{sec:conc}

We have shown that some specific information about the inertia of the distance squared matrix can be given when there are pendant vertices or vertices of degree 2 whose removal disconnects the graph.  From this fact, we gave an alternate proof to the theorem of \cite{BAPAT2016328} which determines the inertia of the distance squared matrix of a tree.  We were able to further use these tools to prove an analogous theorem for unicyclic graphs whose cycle is even and have one tree coming off of a single vertex.  We have also used these tools for a few other specific families of unicyclic graphs.  

Examples in Section \ref{sec:ex} show us that our main theorem cannot generalize to arbitrary unicyclic graphs.  However, it is possible that general bounds could be proven.  Based on results from Section
\ref{subsec:manypendant}, we see the $i_-(\Delta)$ can be as small as $\ell$, the number of leaves, and the largest $i_-(\Delta)$ we have seen is  $\ell+q$ where $q$ is the number of negative eigenvalues coming from the cycle itself.  Thus we could conjecture that
\[
\ell\leq i_-(\Delta)\leq \ell+q.
\]
Possible directions for further research could include proving these bounds for general unicyclic graphs, or finding explicit values for other specific families of unicyclic graphs.  Section \ref{subsec:opp} shows that the question of finding $i_0(\Delta)$ can have many subtleties that could make a general theorem difficult.  This is another possible avenue for future research.

\bibliographystyle{plain}
\bibliography{refs}

\end{document}